\RequirePackage[l2tabu, orthodox]{nag}
\documentclass{article}

\usepackage[utf8]{inputenc}
\usepackage[T1]{fontenc}
\usepackage[english]{babel}
\usepackage{csquotes}
\usepackage{microtype}
\usepackage{lmodern}
\usepackage[a4paper]{geometry}
\usepackage{amsmath}
\usepackage{amssymb}
\usepackage[section]{placeins} 
\usepackage{hyperref}
\usepackage{graphicx}
\usepackage{subcaption}
\usepackage{booktabs}
\usepackage{array}
\usepackage{multirow}
\usepackage{pgfplots}
\usepackage{tikz}
\usepackage{tikzscale}
\usepackage{tensor}
\usepackage{siunitx}
\usepackage{mathtools}
\usepackage{dsfont} 
\usepackage{stackengine}
\usepackage{relsize}
\usepackage{algorithmic}
\usepackage{suffix} 
\usepackage{mleftright}
\usepackage{fancyhdr}
\usepackage{authblk}
\usepackage[algoruled,vlined]{algorithm2e}

\definecolor{TolDarkPurple}{HTML}{332288}
\definecolor{TolDarkBlue}{HTML}{6699CC}
\definecolor{TolLightBlue}{HTML}{88CCEE}
\definecolor{TolLightGreen}{HTML}{44AA99}
\definecolor{TolDarkGreen}{HTML}{117733}
\definecolor{TolDarkBrown}{HTML}{999933}
\definecolor{TolLightBrown}{HTML}{DDCC77}
\definecolor{TolDarkRed}{HTML}{661100}
\definecolor{TolLightRed}{HTML}{CC6677}
\definecolor{TolLightPink}{HTML}{AA4466}
\definecolor{TolDarkPink}{HTML}{882255}
\definecolor{TolLightPurple}{HTML}{AA4499}
\definecolor{InfernoMax}{rgb}{0.988362,0.998364,0.644924}
\definecolor{InfernoMin}{rgb}{0,0,0}
\colorlet{femFG}{TolDarkRed}
\colorlet{femBG}{TolLightRed}
\colorlet{mslrmFG}{TolDarkBlue}
\colorlet{mslrmBG}{TolLightBlue}
\hypersetup{colorlinks=true,
  linkcolor=TolDarkBlue, 
  citecolor=TolLightGreen, 
  urlcolor=TolDarkRed}

\pgfplotsset{compat=newest,
  FEMplot/.style={color=femFG,mark=x,mark options={solid,femFG}},
  MsLRMplot1/.style={color=mslrmFG,mark=*,mark options={solid,mslrmFG}},
  MsLRMplot2/.style={color=mslrmFG,mark=o,densely dotted,mark options={solid,mslrmFG}},
  MsLRMplot3/.style={color=mslrmFG,mark=diamond*,dashed,mark options={solid,mslrmFG}},
}
\pgfkeys{/pgf/number format/.cd,1000 sep={\thinspace}} 
\newlength{\pictureDefaultHeight}
\newlength{\plotDefaultHeight}
\usepackage[backend=biber,
hyperref=auto,
style=numeric,
citestyle=numeric-comp,
sorting=nyt,
block=space,
maxbibnames=9,
maxcitenames=2
]{biblatex}
\DeclarePairedDelimiter{\ceil}{\lceil}{\rceil}
\DeclarePairedDelimiter{\floor}{\lfloor}{\rfloor}
\DeclarePairedDelimiter{\abs}{\lvert}{\rvert}
\DeclarePairedDelimiter{\norm}{\lVert}{\rVert}
\DeclarePairedDelimiterX{\innerp}[2]{\langle}{\rangle}{#1\,\delimsize\vert\,\mathopen{}#2}
\DeclareMathOperator*{\essinf}{Ess\,Inf}
\DeclareMathOperator*{\esssup}{Ess\,Sup}
\DeclareMathOperator*{\vspan}{span}
\DeclareMathOperator*{\vdim}{dim}

\DeclareMathOperator*{\argmax}{argmax}
\DeclareMathOperator*{\cost}{cost}
\DeclareMathOperator*{\grad}{\nabla}
\newcommand{\diver}{\nabla\cdot}
\newcommand{\setN}{\ensuremath{\mathbb{N}}}
\newcommand{\setNsp}{\ensuremath{\mathbb{N}_{*}}} 
\newcommand{\setZ}{\ensuremath{\mathbb{Z}}}
\newcommand{\setR}{\ensuremath{\mathbb{R}}}
\newcommand{\setRsp}{\ensuremath{\mathbb{R}_{+*}}} 
\newcommand{\Id}{\ensuremath{\mathrm{Id}}}
\newcommand{\rrharpoons}{\stackunder[-2pt]{\ensuremath{\rightharpoonup}}{\ensuremath{\rightharpoondown}}}
\newcommand{\Hconv}{\ensuremath{\rrharpoons}}
\newcommand{\wconv}[1]{\ensuremath{\xrightharpoonup{#1}}}
\newcommand{\pderiv}[2]{\frac{\partial\,#1}{\partial\,#2}}
\newcommand{\Zint}[2]{\ensuremath{\mleft\{#1,\dots,#2\mright\}}}
\newcommand{\expec}[1]{\ensuremath{\mathds{E}(#1)}}
\WithSuffix\newcommand\expec*[1]{\ensuremath{\mathds{E}\mleft(#1\mright)}}
\newcommand{\var}[1]{\ensuremath{\operatorname{\mathbb{V}ar}(#1)}}
\WithSuffix\newcommand\var*[1]{\ensuremath{\operatorname{\mathbb{V}ar}\mleft(#1\mright)}}
\newcommand{\covar}[1]{\ensuremath{\operatorname{\mathbb{C}ov}(#1)}}
\WithSuffix\newcommand\covar*[1]{\ensuremath{\operatorname{\mathbb{C}ov}\mleft(#1\mright)}}
\newcommand{\corr}[1]{\ensuremath{\operatorname{corr}(#1)}}
\WithSuffix\newcommand\corr*[1]{\ensuremath{\operatorname{corr}(#1)}}
\newcommand{\deter}[1]{\ensuremath{\operatorname{det}(#1)}}
\WithSuffix\newcommand\deter*[1]{\ensuremath{\operatorname{det}\mleft(#1\mright)}}
\newcommand{\rank}[1]{\ensuremath{\operatorname{rank}(#1)}}
\WithSuffix\newcommand\rank*[1]{\ensuremath{\operatorname{rank}\mleft(#1\mright)}}
\newcommand{\wikns}{\ensuremath{\epsilon}}
\newcommand{\mapK}{\ensuremath{\tilde{K}}}
\newcommand{\qpK}{\ensuremath{\tilde{K}}}
\newcommand{\rvi}{\ensuremath{\tilde{X}}} 
\makeatletter
\newcommand{\myas}{a.s\@ifnextchar.{}{.\@}}
\newcommand{\myae}{a.e\@ifnextchar.{}{.\@}}
\newcommand{\myiid}{i.i.d\@ifnextchar.{}{.\@}}
\newcommand{\myie}{i.e\@ifnextchar.{}{.\@}}
\newcommand{\myIe}{I.e\@ifnextchar.{}{.\@}}
\newcommand{\myeg}{e.g\@ifnextchar.{}{.\@}}
\newcommand{\myviz}{viz\@ifnextchar.{}{.\@}}
\makeatother

\usepackage{amsthm}
\theoremstyle{plain}
\newtheorem{proposition}{Proposition}
\newtheorem{theorem}{Theorem}
\theoremstyle{definition}
\newtheorem{definition}{Definition}
\theoremstyle{remark}
\newtheorem{remark}{Remark}

\numberwithin{equation}{section}
\usepackage[nameinlink]{cleveref}
\crefformat{footnote}{#2\footnotemark[#1]#3} 
\crefformat{equation}{\textup{#2(#1)#3}} 
\Crefformat{equation}{#2Equation~\textup{(#1)}#3} 
\crefname{subsection}{subsection}{subsections}
\Crefname{subsection}{Subsection}{Subsections}
\crefname{figure}{figure}{figures} 
\crefname{page}{p.}{pp.} 
\crefrangeformat{equation}{#3(#1)#4--#5(#2)#6}
\crefmultiformat{equation}{#2(#1)#3}{ and #2(#1)#3}
 {, #2(#1)#3}{ and #2(#1)#3}
\crefrangemultiformat{equation}{#3(#1)#4--#5(#2)#6}{ and #3(#1)#4--#5(#2)#6}{, #3(#1)#4--#5(#2)#6}{ and #3(#1)#4--#5(#2)#6}

\title{\bfseries\LARGE Tensor-based numerical method for stochastic homogenisation}
\author[1]{Quentin Ayoul-Guilmard\thanks{quentin.ayoul-guilmard@centraliens-nantes.net}}
\author[2]{Anthony Nouy\thanks{anthony.nouy@ec-nantes.fr}}
\author[1]{Christophe Binetruy}
\affil[1]{\'Ecole Centrale de Nantes, GeM UMR CNRS 6183, Nantes, France}
\affil[2]{\'Ecole Centrale de Nantes, LMJL UMR CNRS 6629, Nantes, France}
\date{}

\begin{document}

\maketitle

\begin{abstract}
  This paper addresses the complexity reduction of stochastic homogenisation of a class of random materials for   a stationary   diffusion equation. A cost-efficient approximation of the correctors is built using a method designed to exploit quasi-periodicity. Accuracy and cost reduction are investigated for local perturbations or small transformations of periodic materials as well as for materials with no periodicity but a mesoscopic structure, for which the limitations of the method are shown. Finally, for materials outside the scope of this method, we propose to use the approximation of homogenised quantities as control variates for variance reduction of a more accurate and costly Monte Carlo estimator   (using a multi-fidelity Monte Carlo method). The resulting cost reduction is illustrated in a numerical experiment with a control variate from weakly stochastic homogenisation for comparison, and the limits of this variance reduction technique are tested on materials without periodicity or mesoscopic structure.
\end{abstract}


\medskip{}\noindent\textbf{2010 Mathematics subject classification:}   15A69, 35B15, 35B27, 35R60, 65N30.

\section{Introduction}
\label{sec:intro}

Many industries show a growing interest in heterogeneous materials designed with a regular---even periodic---microscopic structure, such as most composite materials.
Prediction of the behaviour of such materials, either during production or intended use, raises various challenges.
Amongst those is the computational complexity induced by the scale difference between the size of the final component and the micro-structure that has to be represented.
Multiscale complexity is a long-standing topic for which various well documented methods such as Multiscale Finite Elements Method~\parencite[MsFEM, explained in][]{Efendiev2009a}, Heterogeneous Multiscale Method~\parencite[HMM, see][]{Abdulle2012} have been developed.
Unlike those general methods, a method designed specifically to break this complexity by exploiting quasi-periodicity was recently proposed by~\textcite{Ayoul-Guilmard2018a}.

Another long-standing approach to tackle this issue is to substitute a homogenised material in the simulations; even though some detailed information is lost, this is still a relevant strategy for many quantities of interest, as discussed by~\textcite[in][]{Murat1997}.
Although low for periodic materials, the homogenisation cost rises significantly with random periodicity loss  \parencite[see \myeg{}][]{Jikov1994}.
This is a barrier to random defect modelling in these simulations.

We propose in this paper to tackle this issue with an approximation of the corrector functions, the computation of which is the costliest part of the homogenisation process, using the method from~\textcite{Ayoul-Guilmard2018a}. \Cref{sec:homog-bg}
will introduce the necessary background on homogenisation for heterogeneous diffusion.
\Cref{sec:tens-struct-mult} will present the key points of the approximation method and illustrate its performance and limits on various numerical experiments.In order to improve on the accuracy of this surrogate model, we propose in \cref{sec:var-reduction} to combine it with another, costlier and more precise, within a multi-fidelity Monte Carlo method \parencite[such as described in][]{Peherstorfer2018a}. Its performance is compared with another surrogate model from~\textcite{Legoll2015a}.

\section{Periodic and stochastic homogenisation}
\label{sec:homog-bg}

\subsection{General homogenisation}
\label{sec:gen-homog}

Let $D\subset\setR^d$ an open bounded domain. We consider an heterogeneous stationary diffusion problem
\begin{align}
  \label{eq:diff-strong}
  -\diver(K \grad u) & = f, \\
  \label{eq:diff-bc}
  u_{|\partial D} & = 0,
\end{align}
with $f\in H^{-1}(D)$.
$K$ is assumed to have variations at a small scale on $D$, which makes practical computations expensive. We wish to substitute to~\cref{eq:diff-strong} a problem without such variations that preserves certain quantities of interest. 
For any $0<\alpha<\beta$ we let,
\begin{multline*}
  \mathcal{M}(\alpha,\beta) := 
  \bigl\{
  a\in L^\infty(D;\setR^{d\times d}) : \innerp{a(x)y}{y} \geqslant \alpha\|y\|^2 \\
  \text{ and } \|a(x)y\| \leqslant \beta\|y\|, \forall y\in \setR^d, \text{ for \myae{} } x\in D
  \bigr\}.
\end{multline*}
\begin{definition}[H-convergence]
  \label{def:h-conv}
  Let $f\in H^{-1}(D)$, $(K_n)_{n\in\setN} \in \mathcal{M}(\alpha,\beta)^\setN$ and  $K_\star \in \mathcal{M}(\alpha',\beta')$ with $0<\alpha<\beta$ and $0<\alpha'<\beta'$.
  Let $u_\star\in H^1(D)$ be solution to
  \begin{gather}
    \begin{dcases}
      \label{eq:diff-homog}
      -\diver(K_\star\grad u_\star) = f  \\ 
      \cref{eq:diff-bc}
    \end{dcases}	
  \end{gather}
  and, for any $n\in\setN$, let $u_n\in H^1(D)$ be solution to
  \begin{gather*}
    \begin{dcases} 
      -\diver(K_n\grad u_n) = f \\ 
      \cref{eq:diff-bc} 
    \end{dcases}
  \end{gather*}
  We say that $(K_n)_{n\in\setN}$ H-converges to $K_\star$ and note $K_n \Hconv K_\star$ when
  \begin{gather*}
    u_n \wconv{H^1(D)} u_\star ,\\ 
    K_n\grad u_n \wconv{L^2(D;\setR^d)} K_\star\grad u_\star .
  \end{gather*}
\end{definition}
H-convergence is a generalisation by~\textcite{Murat1997} of the G-convergence to non-symmetric operators. G-convergence was originally introduced by~\textcite{Spagnolo1968}. The reader may consult~\textcite{Defranceschi1993} for a detailed explanation of both.

\begin{theorem}[Fundamental theorem of homogenisation]
  \label{prop:fth}
  Let
  $(K_n)_{n\in\setN}\in \mathcal{M}(\alpha,\beta)^\setN$ with $0<\alpha<\beta$. 
  There exists a subsequence $(K_{s(n)})_{n\in\setN}$ and an  operator $K_\star \in\mathcal{M}(\alpha,\beta^2\alpha^{-1})$ such that
  \begin{gather*}
    \begin{array}{>{\displaystyle}r>{\displaystyle}c>{\displaystyle}l}
      K_{s(n)} & \Hconv & K_\star , \\
      K_{s(n)}\grad u_{s(n)}\cdot\grad u_{s(n)} & \wconv{} & K_\star\grad u_\star\cdot\grad u_\star \text{ weakly-$\ast$ in } \mathcal{D}'(\setR^d) ,\\
      \text{and } \int_D K_{s(n)}\grad u_{s(n)}\cdot\grad u_{s(n)} & \rightarrow & \int_D K_\star\grad u_\star\cdot\grad u_\star,      
    \end{array}
  \end{gather*}
  where $\mathcal{D}'(\setR^d)$ denotes the space of distributions on $\setR^d$. $K_\star$ is commonly called the \enquote{homogenised operator} of $K$. 
\end{theorem}
\begin{proof}
  See~\textcite[pp. 31--36]{Murat1997}.
\end{proof}

\begin{theorem}[Corrector theorem]
  \label{prop:corrector}
  Let $(K_n)_{n\in\setN}\in \mathcal{M}(\alpha,\beta)^\setN$ with $0<\alpha<\beta$.
  For any $i\in\Zint{1}{d}$, there exists a sequence $(w^i_n)_{n\in\setN} \in H^1(D)^\setN$ such that
  \begin{align}
    w^i_{s(n)} & \wconv{H^1(D)} 0 
                 , \notag\\
    -\diver\mleft(K_{s(n)}(e_i+\grad w^i_{s(n)})\mright) & \xrightarrow{ H^{-1}(D)} 0
                                                         , \notag\\
    \label{eq:tc3}
    \text{and } \grad u_{s(n)} - (\Id+\grad w_{s(n)})\grad u_\star & \xrightarrow{L^1(D)^d} 0
  \end{align}
  where $s$ defines the same subsequence as in \cref{prop:fth}.
  These functions $w^i_n$ are called \enquote{correctors} for they provide the necessary correction to $\grad u_{s(n)}$, in the sense of \cref{eq:tc3}, to converge strongly towards $\grad u_\star$. Additionally, they can be used to deduce $\grad u_{s(n)}$ from $\grad u_\star$.
  A direct consequence is that
  \begin{gather*}
    K_{s(n)}(\Id+\grad w_{s(n)}) \wconv{L^2(D)^{d\times d}} K_\star .
  \end{gather*}
\end{theorem}
\begin{proof}
  See~\textcite[pp. 37--42]{Murat1997}.
\end{proof}

\subsection{Periodic homogenisation}
\label{sec:per-homog}

The corrector functions introduced above are used to build the homogenised operator, but more information on $K$ is required to give an expression of $K_\star$. The quasi-periodicity we are interested in arises most often as the result of random perturbation of a periodic medium, and both the periodic and stochastic context are well documented. Let us first consider the case where $K$ is a periodic operator.

\begin{proposition}[Periodic homogenisation]
  \label{prop:homog-per}
  
  Let $K\in \mathcal{M}(\alpha,\beta)$, with $0<\alpha<\beta$, be periodic of period $Y\subset\setR^d$. 
  Let $(\epsilon_n)_{n\in\setN}\in] 0,+\infty{} [^\setN$ be a sequence converging to \num{0}.
  Then the results from \cref{prop:fth,prop:corrector} hold   for   the sequence $(K(\frac{\cdot}{\epsilon_n}))_{n\in\setN}$ itself---not only on a subsequence. Furthermore,
  \begin{gather}
    \label{eq:homog-per}
    \forall(i,j)\in\Zint{1}{d}^2,\; (K_\star)_{ij} = \int_{Y} (e_i+\grad w_i)\cdot K e_j ,
  \end{gather}
  where the corrector $w_i\in H^1_\sharp(Y):=\{v\in H^1(Y) : v \text{ is } Y\text{-periodic}\}$ solves
  \begin{gather}
    \label{eq:corr-per}
    -\diver\mleft(K(e_i+\grad w_i)\mright) = 0 .
  \end{gather}
\end{proposition}
\begin{proof}
  This result has been proven by different methods: first in~\citeyear{Murat1978a} through compensated compactness by~\textcite{Murat1978a}, then in~\citeyear{Tartar1979a} with oscillating test functions by~\textcite{Tartar1979a} and later in~\citeyear{Nguetseng1989a} via two-scale convergence by~\textcite{Nguetseng1989a}.
\end{proof}

The additional periodicity hypothesis in \cref{prop:homog-per} gives us expression~\cref{eq:homog-per} to compute the homogenised operator $K_\star$. The $d$ corrector problems~\cref{eq:corr-per} to be solved are limited to the period $Y$, hence the reasonable cost. Additionally, this homogenisation is independent of the source term $f$ and the boundary conditions, since they do not affect the homogenised operator, and are identical in both the original problem~\crefrange{eq:diff-strong}{eq:diff-bc} and the homogenised one~\cref{eq:diff-homog}. This justifies the concept of \enquote{homogenised material}.

\subsection{Stochastic homogenisation}
\label{sec:stoch-homog}

In the case where $K$ is stochastic, the general results from \cref{prop:fth,prop:corrector} hold almost surely. We introduce a probability space  $(\Omega,\mathcal{E},\mathbb{P})$.
Stochastic homogenisation generally relies on a stationarity assumption, \myie{} invariance of probability laws with respect to space (or time) translation. We expect $K$ to have a spatial quasi-periodic structure, therefore we choose a definition of stationarity suited to represent such structure.

\begin{definition}[Discrete stationarity]
  \label{def:stationarity}
  Let us assume that the group $(\setZ^d,+)$ acts on $(\Omega,\mathcal{E},\mathbb{P})$ through an ergodic measure-preserving transformation $\tau=(\tau_z)_{z\in\setZ^d}$ :
  \begin{gather*}
    \forall e \in\mathcal{E},\quad \mleft[\forall z\in\setZ^d, \tau_ze=e\mright] \rightarrow \mleft[\mathbb{P}(e)\in\{0,1\}\mright] .
  \end{gather*}
  Any   $\phi\in L^1_{loc}(\setR^d;L^1(\Omega))$
    is said to be (discretely) stationary if, and only if,
  \begin{gather*}
    \forall z\in\setZ^d, \phi(x+z,\omega) = \phi(x,\tau_z\omega) \text{ \myae{} on } \setR^d\times\Omega .
  \end{gather*}
\end{definition}
We can observe that the probability law of a function $\phi$, stationary in the sense of \cref{def:stationarity}, is invariant by any translation by $z\in\setZ^d$: since $\tau_z$ is preserves the measure, $\phi(x,\cdot)$ and $\phi(x+z,\cdot)$ follow the same probability law.
This is related to periodicity, \myeg{} any stationary function has a $Y$-periodic expectation, where $Y:=[0,1]^d$.
Up to a linear mapping $x\mapsto x/\varepsilon$, we can relate any $[0,\varepsilon]^d$-periodic function to \cref{def:stationarity}.
It should be noted that discrete stationarity is not a particular case of continuous stationarity.

Henceforth, stationarity will be meant in the discrete sense defined above and $K$ is assumed to be stationary.
This definition, albeit not used directly here, allows us to use some well-known results of stochastic homogenisation, such as the ergodic theorem formulated below in this discrete setting.

\begin{theorem}[Ergodic theorem]
  \label{prop:ergodic}
  Let   $K\in L^\infty(\setR^d;L^1(\Omega))$  be stationary. Then we have almost surely
  \begin{gather*}
    \frac{1}{(2N+1)^d}\sum_{\norm*{z}_\infty\leqslant N} K(x,\tau_z\omega) \xrightarrow[N\rightarrow\infty]{L^\infty(\setR^d)} \expec*{K(x,\cdot)} \\
    \text{and } K\mleft(\frac{x}{N},\omega\mright) \underset{N\rightarrow\infty}{\wconv{L^\infty(\setR^d)}} \expec*{\int_Y K} \quad \text{in weak-$\ast$ topology}    .
  \end{gather*}
\end{theorem}
\begin{proof}
  See~\textcite{Birkhoff1931}.
\end{proof}

\begin{proposition}[Stochastic homogenisation]
  \label{prop:homog-sto}
  Let $K\in L^\infty(\setR^d;L^1(\Omega))$ be stationary and such that $K(\cdot,\omega)\in \mathcal{M}(\alpha,\beta)$, with $0<\alpha<\beta$, almost surely.
  Let $(\epsilon_n)_{n\in\setN}\in]0,+\infty[^\setN$ be a sequence converging to \num{0}.
  Then the results from \cref{prop:fth,prop:corrector} hold   for
    a subsequence of $(K(\frac{\cdot}{\epsilon_n},\cdot))_{n\in\setN}$. Furthermore,
  \begin{gather*}
    \forall (i,j)\in\Zint{1}{d}^2,\; (K_\star)_{ij} = \expec*{\int_Y (e_i+\grad w_i)\cdot K e_j},
  \end{gather*}
  where%
  \footnote{We let $\mathcal{W} := \mleft\{ v \in L^2_{loc}(\setR^d;L^2(\Omega)) : \grad v \in L^2_{unif}(\setR^d;L^2(\Omega))  \mright\}$ 
    with $L^2_{unif}(\setR^d;L^2(\Omega)) := \bigl\{v \in L^2(\setR^d;L^2(\Omega)) :
    \exists C \in\setRsp, \forall x\in\setR^d\times\setRsp, \norm{v}_{L^2(\mathcal{B}(x,1);L^2(\Omega))}\leqslant C \bigr\}$, 
    and $\mathcal{B}(x,R)$ the ball of centre $x$ and radius $R$.}
   every corrector $w_i\in\mathcal{W}$ is solution to
  \begin{gather}
    \label{eq:corr-sto}
    \begin{dcases}
      -\diver\mleft(K(e_i+\grad w_i)\mright) & = 0 \text{ \myas{} on } \setR^d \\
      \expec*{\int_Y \grad w_i}             & = 0                          \\
      \grad w_i \text{ stationary}         & 
    \end{dcases}.
  \end{gather}
\end{proposition}
\begin{proof}
  See~\textcite[th. 7.9, p. 244]{Jikov1994}.
\end{proof}

Unlike the periodic corrector problems~\cref{eq:corr-per}, the stochastic corrector problems~\cref{eq:corr-sto} are impractical to solve because of the global conditions over both $\setR^d$ and $\Omega$. This is why practical computations of $K_\star$ usually introduce the \enquote{apparent homogenised operator}
\begin{gather}
  \label{eq:app-homog-sto}
  (K_\star^N)_{ij} := \int_{Y_N} (e_i+\grad w_{N,i})\cdot K\cdot e_j
\end{gather}
on a truncated domain $Y_N = [0,N]^d$, where the \enquote{apparent correctors} are solutions to
\begin{gather}
  \label{eq:app-corr-sto}
  \begin{dcases}
    -\diver(K(e_i+\grad w_{N,i})) & = 0 \quad \text{\myas{} on } Y_N \\
    w_{N,i} ~Y_N\text{-periodic} &
  \end{dcases}
\end{gather}
Since $K$ is stationary, \cref{prop:ergodic} yields~\parencite[see][th. 1 p. 158]{Bourgeat2004}
\begin{gather}
  \expec*{K_\star^N} = \lim_{N\rightarrow\infty} K_\star^N = K_\star,
\end{gather}
therefore we can approximate $K_\star$ with a Monte Carlo estimator
\begin{gather}
  \label{eq:mc-estimator}
  \theta_m(K_\star^N) := \frac{1}{m}\sum_{i=1}^{m}K_\star^N(\cdot,\omega_i)
\end{gather}
of $\expec{K_\star^N}$, using $m$ samples (independent and identically distributed) $K(\cdot,\omega_i), 1\leqslant i\leqslant m$. The chain of approximation thus reads $K_\star \approx \expec{K^N_\star} \approx \theta_m(K^N_\star)$. An evaluation of this estimator requires $d\times m$ resolutions of the apparent corrector problem~\cref{eq:app-corr-sto}, which remains considerably more expensive than periodic homogenisation. Furthermore,
\begin{gather}
  \label{eq:mc-conv-rate}
  \norm*{\expec*{K^N_\star}-\theta_m(K^N_\star)}_{L^2(\Omega)} = \sqrt{\frac{\var*{K^N_\star}}{m}},
\end{gather}
so the rate of convergence of the estimator is rather slow; hence the benefit of a method to exploit quasi-periodicity in order to break down the complexity of these corrector problems.

\subsection{Homogenisation with random mapping}
\label{sec:homog-map}

Alternatively to the previous setting, one may consider the problem~\cref{eq:diff-strong} to be set on a domain transformed by a random stationary diffeomorphism $\phi(\omega):\setR^d\rightarrow\setR^d$ such that $K\circ\phi=:K_\sharp$ is $Y$-periodic. We can then write~\cref{eq:diff-strong} as
\begin{gather}
  \label{eq:diff-sto-map}
  -\diver \mleft( K_\sharp\circ\phi^{-1}(\omega) \grad u \mright) = f.
\end{gather}
This domain transformation is illustrated on \cref{fig:mapping}.

\begin{figure}[hbp]
  \centering
  \includegraphics[width=.8\linewidth]{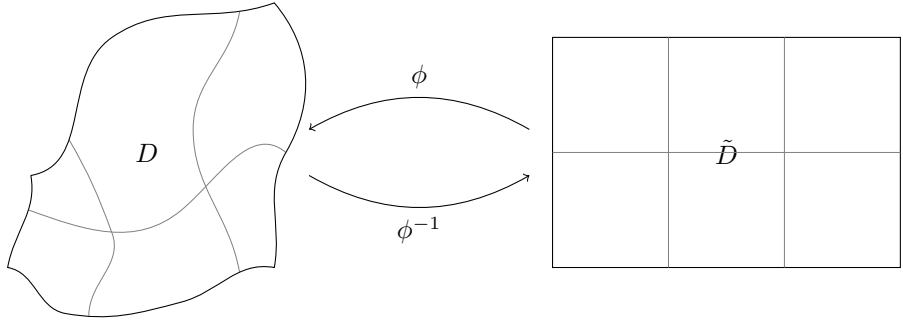}
  \caption{Random mapping}
  \label{fig:mapping}
\end{figure}

\begin{definition}[Random stationary diffeomorphism]
  \label{defi:random-stat-diff}
  For $0<\mu\leqslant\nu<+\infty$, let us note $\Phi(\mu,\nu)$ the set of functions   $\varphi\in L^1(\Omega;C^1(\setR^d,\setR^d))$
    such that, $\forall\omega\in\Omega,\; \varphi(\omega)$ is \myas{} a diffeomorphism, $\grad\varphi:\Omega\to C^0(\setR^d;\setR^{d\times d})$ (such that $\grad\varphi(\omega)$ is the gradient of $\varphi(\omega)$) is stationary,
  \begin{gather*}
    \essinf_{(x,\omega)\in\setR^d\times\Omega} \deter*{\grad\varphi(\omega)(x)} \geqslant \mu, \qquad\text{ and }\qquad
    \esssup_{(x,\omega)\in\setR^d\times\Omega} \norm*{\grad\varphi(\omega)(x)} \leqslant \nu.
  \end{gather*}
\end{definition}

\begin{proposition}[Homogenisation with random mapping]
  \label{prop:homog-map}
  Let $K_\sharp\in \mathcal{M}(\alpha,\beta)$, with $0<\alpha<\beta$, be periodic of period $Y\subset \setR^d$; let $\phi\in\Phi(\mu,\nu)$, with   $0<\mu\leqslant\nu<+\infty$; let $(\epsilon_n)_{n\in\setN}\in]0,+\infty[^\setN$ be a sequence converging to \num{0}. For any $n\in\setN$, we note $u_n$ the solution in $H^1(D)$ to
  \begin{gather*}
    \mleft\{
      \begin{aligned}
        -\diver \mleft( K_\sharp\circ\phi(\omega)^{-1}\mleft(\frac{\cdot}{\epsilon_n}\mright) \grad u_n \mright) & = f \\
        u_{n|\partial D} & = 0
      \end{aligned}
    \mright..
  \end{gather*}

  Then $u_n\wconv{H^1(D)}u_\star$ and $u_n\xrightarrow{L^2(D)}u_\star$, \myas{}, where $u_\star$ is solution of the homogenised problem
  \begin{gather*}
    -\diver \mleft( K_\star \grad u_\star \mright) = f.
  \end{gather*}
  The homogenised operator's coefficients are defined as
  \begin{gather*}
    (K_\star)_{ij} = \deter*{\expec*{\int_Y \grad\phi(\cdot)}}^{-1}
    \expec*{
      \int_{\phi(\cdot)(Y)} (e_i+\grad w_i)\cdot (K_\sharp\circ\phi(\cdot)^{-1}) e_j
    }
  \end{gather*}
  for all $(i,j)\in\Zint{1}{d}^2$. Each corrector $w_i$ is the solution of
  \begin{gather}
    \label{eq:corr-map}
    \begin{dcases}
      -\diver \mleft( K_\sharp\circ\phi(\cdot)^{-1} (e_i+\grad w_i) \mright) = 0 \\
      \expec*{\int_{\phi(\cdot)(Y)} \grad w_i} = 0 \\
      \grad \tilde{w_i} \text{ is stationary}\quad (\tilde{w}_i\circ\phi(\cdot)^{-1} := w_i)
    \end{dcases}
  \end{gather}
  in   $H^1(Y;L^2(\Omega))$, unique up to the addition of a random constant.
\end{proposition}
\begin{proof}
  See~\textcite[§ 1.3, pp. 719--722]{Blanc2006}.
\end{proof}

\begin{remark}[Extension of \cref*{prop:homog-map}]
  \label{rmk:homog-map-stat}
  Different assumptions on the operator in~\cref{eq:diff-sto-map} may be considered. Results similar to \cref{prop:homog-map} have been proven by~\textcite[pp. 11--13]{Blanc2007}, with assumption of ergodicity and stationarity, either continuous or discrete in the sense of \cref{def:stationarity}. Let us recall that the periodic setting considered here is a particular case of the discrete stationary setting, but \emph{not} of the continuous one.
\end{remark}

We can see that problem~\cref{eq:corr-map} presents the same issues as problem~\cref{eq:corr-sto} where practical computations are concerned. The previous method based on apparent correctors computed on truncated domains~\parencite[detailed in][]{Bourgeat2004} has been adapted to this setting by \textcite{Costaouec2010a}. We introduce the apparent homogenised operator 
\begin{multline*}
  (K_\star^N(\omega))_{ij} := \deter*{\frac{1}{|Y_N|} \int_{Y_N} \grad \phi(\omega)}^{-1} \\
  \times  \frac{1}{|Y_N|}\int_{Y_N} e_iK_\sharp\cdot(e_j+ \tensor[^t]{(\grad\phi(\omega))}{^{-1}} \grad\tilde{w}_j^N(\omega)) \deter*{\grad\phi(\omega)},
\end{multline*}
computed from apparent correctors $\tilde{w}_i^N\in H^1_\sharp(Y_N)$. These are the solutions, unique up to an additive random constant, to
\begin{align}
  \int_{Y_N} \grad\tilde{v} \grad\phi^{-1}\cdot K (e_i+ \tensor[^t]{\grad\phi}{^{-1}}\grad\tilde{w}_i^N) \deter*{\grad\phi} &  = 0, \forall\tilde{v}\in H^1_\sharp(Y_N) , \notag\\
  \label{eq:app-corr-map}
  \text{which is equivalent to}\quad \int_{Y_N} \grad\tilde{v} \cdot\mapK (\tensor[^t]{\grad\phi}{}e_i+ \grad\tilde{w}_i^N) &  = 0, \forall\tilde{v}\in H^1_\sharp(Y_N) , 
\end{align}
with $\mapK := \deter{\grad\phi} \tensor*[]{\grad\phi}{^{-1}} K_\sharp \tensor[^t]{\grad\phi}{^{-1}}$.

It has been proven by~\textcite{Legoll2014b} that $K^N_\star \xrightarrow[N\rightarrow\infty]{\text{\myas{}}} K_\star$ so we can use a Monte Carlo method to estimate $\expec{K^N_\star}$ and therefore to get an approximation of
$K_\star$, as in the previous setting.

\section{Tensor-based two-scale method}
\label{sec:tens-struct-mult}

\subsection{Description of the method}
\label{sec:lrm}

We will briefly outline a method to get a cost-efficient approximation to the solution of stationary heterogeneous diffusion problems such as~\cref{eq:app-corr-sto}. This method has been introduced by~\textcite{Ayoul-Guilmard2018a}.

We consider a multiscale representation which separates microscopic and mesoscopic information, the mesoscopic scale being that of the periods. The domain $D$ is partitioned into cells $D_i$, each being a period of the reference periodic medium. Letting $I\subset\setN$ be the cell index set and $Y\subset\setR^d$ a reference period, $I\times Y$ is identified with $\bigcup_{i\in I}D_i$ (see \cref{fig:domain-separation}) through the isomorphism $\zeta:I\times Y \ni (i,y) \mapsto y+b_i \in D$, with $b_i$ such that $Y+b_i=D_i$ for any cell $D_i, i\in I$.

\begin{figure}
  \centering
  \includegraphics[height=\pictureDefaultHeight]{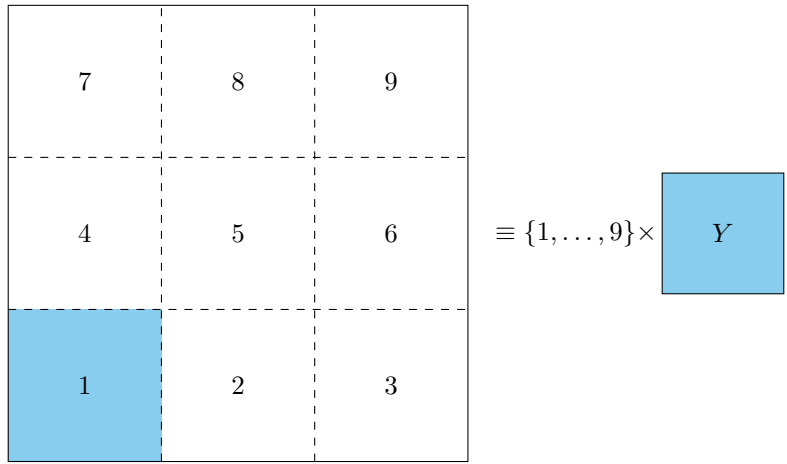}
  \caption{$D = \bigcup_{i\in I}D_i \equiv I \times Y$}
  \label{fig:domain-separation}
\end{figure}

Likewise, $\setR^D$ is isomorphic to $\setR^I\otimes\setR^Y$ with isomorphism $\Upsilon$ such that, for any $v\in\setR^D$, we identify $v$ with
\begin{gather}
  \label{eq:tensor-format}
  \Upsilon(v) = \sum_{n=1}^{\#I} e_n^I \otimes  \mleft(v\circ\zeta(n,\cdot)\mright) \quad
  \text{\myie{}, } \forall(i,y)\in I\times Y, \; v(\zeta(i,y)) = \sum_{n=1}^{\#I} e_n^I(i)v(\zeta(n,y)),
\end{gather}
where $\{e_n^I\}_{n\in I}$ is the canonical orthonormal basis of $\setR^I$.
Henceforth we will define the rank of $v$ as the rank of $\Upsilon(v)$, which is the minimal $r$ such that $\Upsilon(v)=\sum_{n=1}^r v^I_n \otimes v^Y_n$.
More specifically, $K$ and $f$ can be identified with functions of $\setR^I\otimes L^\infty(Y)$ and $\setR^I\otimes L^2(Y)$, respectively, and the broken Sobolev space $H^1(\bigcup_{i\in I} D_i)= \mleft\{v \in L^2(\mathbb{R}) : \forall i\in I, v_{|D_i}\in H^1(D_i) \mright\}$ is identified with $\setR^I\otimes H^1(Y)$.

Consequently, problem~\cref{eq:app-corr-sto} can be formulated over a space $V(D)\subset\setR^I\otimes H^1(Y)$ within a discontinuous Galerkin setting \footcite[with a SWIP method detailed in][]{DiPietro2011}.
This  function space is defined as $V(D):=\setR^I\otimes V(Y)$ with $V(Y) := \mleft\{v\in H^1(Y) : (\nabla v)_{|\partial Y}\in L^2(\partial Y)^d \mright\}$.
We then introduce a finite-dimensional subspace $V_h(D):=\setR^I\otimes V_h(Y)$ with $V_h(Y)\subset V(Y)$ and look for an approximation of the corrector functions $w_{N,k}$ (with $N:=\#I^{1/d}$) as solutions to
\begin{gather}
  \label{eq:tensor-corr}
  a^{swip}_h(w_{N,k},\delta w) = b_{h,k}(\delta w), \; \forall \delta w \in V_h(D).
\end{gather}

The benefit of format~\cref{eq:tensor-format} lies in the match between the structure of the tensor product space and the underlying periodic organisation. The representation in $V(D)$ of a quasi-periodic function has as low a rank $r$ as it is close to periodic. Note for example that a periodic function $v$ is identified with a rank one tensor $\Upsilon(v)=1^I\circ\varpi$, with $1^I:I\to\{1\}$ and $\varpi=v\circ\zeta(1,\cdot)$.
 The quasi-periodicity assumption on $K$ is therefore interpreted as a low-rank assumption, hence our expectation that the solution of~\cref{eq:tensor-corr} would have a low rank. We exploit this by building a low-rank approximation of the solution of tensor-structured corrector problem~\cref{eq:tensor-corr} using a greedy algorithm~\parencite[see][§ 3.2]{Ayoul-Guilmard2018a}). The residual error is controlled with a tolerance $\varepsilon$ so that
\begin{gather}
  \label{eq:residual-error-crit}
  \frac{\norm*{a^{swip}_h(w_{N,k},\cdot) - b_{h,k}}_{V_h(D)'}}{\norm*{b_{h,k}}_{V_h(D)'}} \leqslant \varepsilon.
\end{gather}

\begin{remark}[Random mapping]
  We assumed so far that the quasi-periodic structure of the medium (\myie{} $K$) matches a perfectly regular mesoscopic mesh. The case where such mesh would be irregular, even random, can be addressed easily if we know a random stationary diffeomorphism $\phi$ such that $K\circ\phi=:K_\sharp$ is
  periodic; then, the results from \cref{sec:homog-map} apply. For this low-rank method to remain efficient, $\mapK$ is assumed to have a low rank in $\setR^I\otimes L^\infty(Y)$, which depends on the ranks of $\grad\phi^{-1}$ and $\deter{\phi}$ (since $\mathrm{rank}(K_\sharp)=1$). \Cref{rmk:homog-map-stat} mentions that similar results can be proven for $K\circ\phi$ stationary, in which case the rank of $\mapK$ is also influenced by the rank of $K\circ\phi$.
\end{remark}

\subsubsection*{Modes recycling}

An additional benefit arises from the tensor structure of $V(D)$ for successive simulations.
One interpretation of format~\cref{eq:tensor-format} is that
\begin{gather*}
  \forall i\in I,\quad v(i,\cdot)\in\vspan\mleft(v^Y_n\mright)_{n\in\Zint{1}{r}},
\end{gather*}
so  any function may be considered cell-wise as a superposition of the same set of microscopic patterns $v^Y_n$ that we will call \enquote{modes}.
This concept is illustrated on \cref{fig:modes-example}: the nine modes in \cref{fig:modes-example_library} compose the function pictured on \cref{fig:modes-example_approx}, where white lines delineate the hundred cells.
Whenever the solutions to successive problems can be expected to evince similar modes, one may recycle those computed from one resolution to the next.

Such situation is common for random quasi-periodic materials, \myeg{} with sampling-based methods such as Monte Carlo estimation.
Indeed, this strategy bears particular relevance if the differences between samples are consistent with the mesoscopic structure, which can be expected from our view of quasi-periodicity as a randomly perturbed periodicity. 
Since the perturbation is such that the rank of $K$ remains low, recurring patterns are expected across realisations of $K$.
Should those realisations differ mostly in their spatial distribution, then the conductivity's microscopic information remains unchanged and---expectedly---so does the solution's.
Particularly, we will present in \cref{sec:lrm-tests} examples of problem~\cref{eq:app-corr-sto} typical of such cases where only the functions $K_n^I\in\setR^I$ in the representation $K\equiv \sum_n K^I_n \otimes K^Y_n $ are random%
\footnote{In this specific case, operators over $V_h(Y)$ could be recycled as well, thus saving the associated assembling cost; since we observed this cost to be insignificant, we did not do so.}.
Furthermore, this modes library is expected to reach a practical completion as a consequence of $K$'s stationarity. \myIe{} the library will eventually become comprehensive enough to  accurately represent the solution associated to almost any realisation of $K$, within the chosen tolerance.

The most straightforward implementation of this recycling strategy would be to set an anterior result $u_0:=\sum_{k=1}^{r_0} u_k^I\otimes u_k^Y$ as initial point for the subsequent resolution.
Although the greedy algorithm would usually see its performance significantly improved by an initialisation close to the solution, it would always compute at least one additional term $u_{r_0+1}^I\otimes u_{r_0+1}^Y$, even if the two successive problems were identical.
The resulting ever-growing library would yield worse performance than the original method.
Therefore, we first look for a suitable approximation on the subspace spanned by the modes computed so far: we compute the Galerkin projection of $u_0$ on $\setR^I\otimes\vspan\{u_1^Y,\ldots,u_{r_0}^Y\}$ then check the resulting approximation against criterion~\cref{eq:residual-error-crit}.
If it is not satisfied, then said projection is set as the greedy algorithm's initial point.
Once the library has reached its aforementioned practical completion, subsequent resolutions are reduced to this initial projection step, \myie{} a Galerkin projection onto a subspace of $V_h(D)$ of which a basis is known and hopefully low-dimensional, thereby further reducing complexity.

This modes recycling strategy is akin to an adaptive reduced basis approach, and it should be noted that this basis construction strategy is non-optimal in the sense that the low-rank approximation is not optimal: the same precision might be achieved with a lower rank and yield a library of fewer modes.
Although this non-optimal greedy algorithm is usually more cost-efficient as far as a single computation is concerned, this strategy has not been specifically designed toward library construction.

\begin{figure}[hbt]
  \centering
  \begin{subfigure}{0.48\linewidth}
    \centering
    \begin{tikzpicture}
      \node[anchor=south west,inner sep=0] (image) at (0,0) {\includegraphics[width=\linewidth,height=\linewidth]{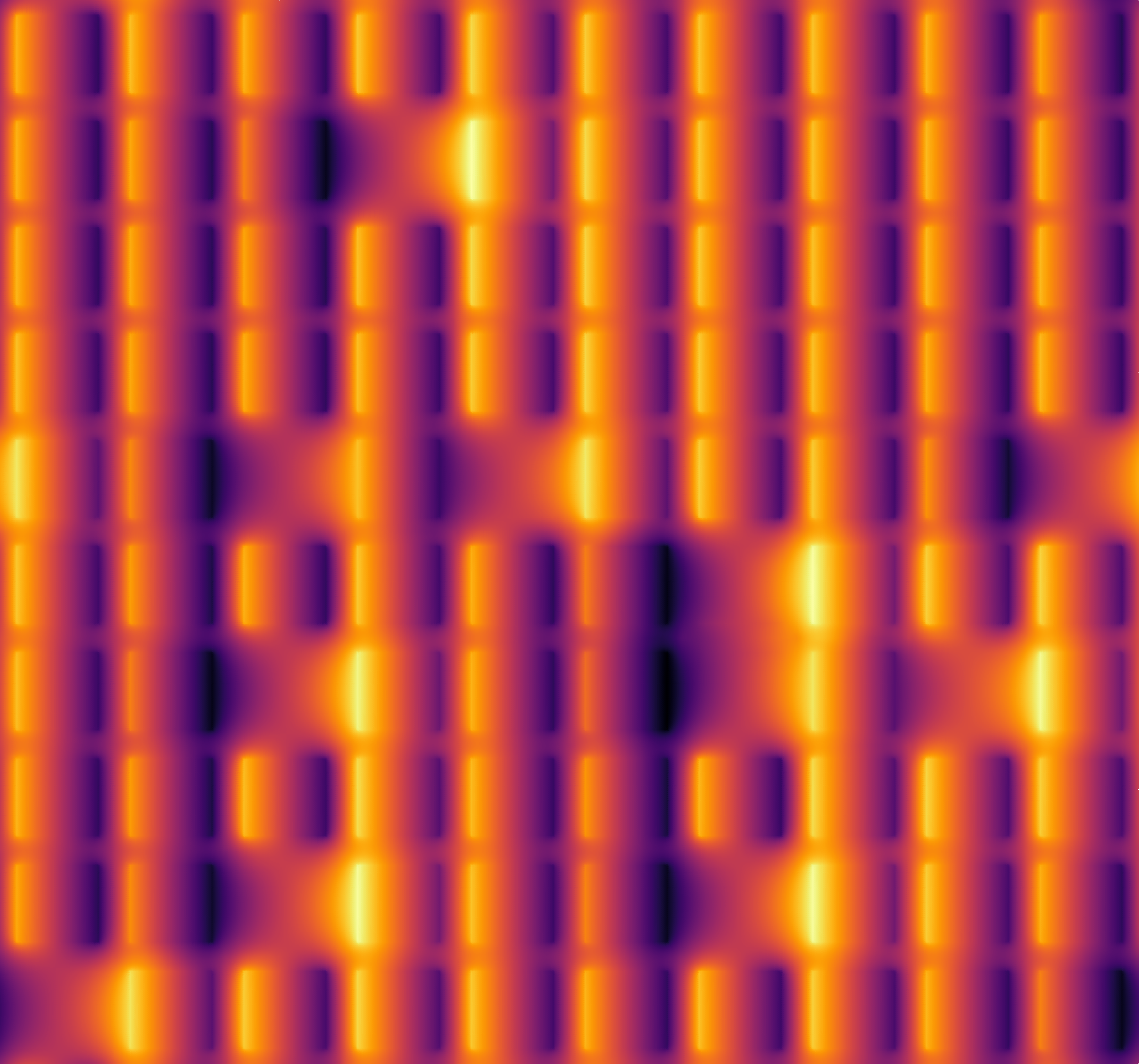}};
      \begin{scope}[x={(image.south east)},y={(image.north west)}]
        \draw [white,opacity=.5,thin] (0,0) grid[step=.1] (1,1);
      \end{scope}
    \end{tikzpicture}
    \caption{Approximation $v_9$}
    \label{fig:modes-example_approx}
  \end{subfigure}
  ~
  \begin{subfigure}{0.48\linewidth}
    \centering
    \includegraphics[width=\linewidth,height=\linewidth]{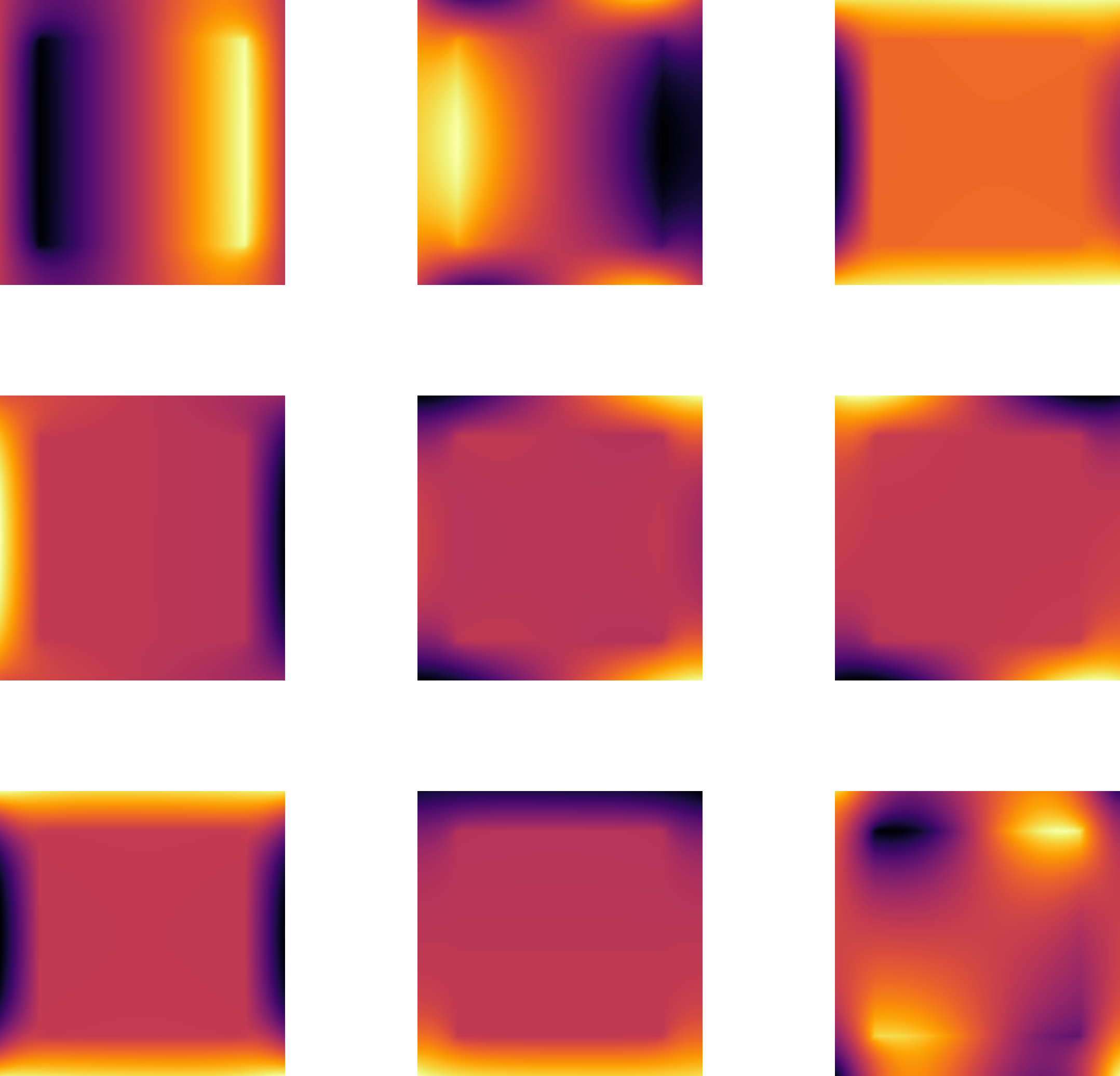}
    \caption{Modes $(v^Y_n)_{n\in\Zint{1}{9}}$}
    \label{fig:modes-example_library}
  \end{subfigure}
  \caption{Modes library example}
  \label{fig:modes-example}
\end{figure}

\subsection{Numerical examples}
\label{sec:lrm-tests}

Each but the last of the following numerical experiments will involve an approximation of the homogenised conductivity $K_\star$ following the process described in \cref{sec:stoch-homog}
with $D\subset\setR^2$. The truncated corrector problems~\cref{eq:app-corr-sto} will be solved using the multiscale low-rank method (MsLRM) outlined in \cref{sec:lrm}, and the homogenised conductivity will be approximated using a Monte Carlo estimator $\theta_m(K^{N}_\star)$ as expressed in equation~\cref{eq:mc-estimator}. The number of samples $m=m(\eta)$ will be selected to satisfy the error criterion
\begin{gather*}
  \norm*{\expec*{K^{N}_\star}-\theta_{m(\eta)}(K^{N}_\star)}_{L^2(\Omega)} \leqslant \eta,
\end{gather*}
using equation~\cref{eq:mc-conv-rate} and an estimation of $\var{\theta_{m(\eta)}(K^{N}_\star)}$. We denote such an estimator $\theta^\eta:=\theta_{m(\eta)}$.
Unless specified otherwise, the default parameters values in \cref{tab:default-param} apply to every experiment. These values will be discussed with experiments results below.

\begin{table}[hbt]
  \centering
  \caption{Default tests parameters}
  \begin{tabular}{cccccc}\toprule
    $N^2$      & $\vdim(V_h(Y))$ & $\varepsilon$  & $p$       & $K_2/K_1$ & $\eta$     \\\midrule
    \num{1600} & \num{441}       & \num{0.01} & \num{0.1} & \num{100} & \num{0.05} \\\bottomrule
  \end{tabular}
  \label{tab:default-param}
\end{table}

The reference periodic medium of our main test case has a square inclusion in each cell, more conductive than the other phase, and the two phases have equal volume fractions. The random perturbation is the absence of some inclusions, replaced with the other phase. The random conductivity can be expressed as
\begin{gather}
  \label{eq:K-bernoulli}
  K(i,y,\omega) = K_1 + B_i(\omega)\chi(y)(K_2 - K_1), \quad \forall(i,y,\omega)\in I\times Y\times\Omega,
\end{gather}
where the random variables $(B_i)_{i\in I}$ follow independent and identical Bernoulli laws with $\mathbb{P}(B_i=1)=p$, and $\chi:Y\to\{0,1\}$ is the characteristic function of the inclusion within the reference cell. $K$ is stationary according to \cref{def:stationarity}. \Cref{fig:bi-periodic} displays examples of such conductivity, associated correctors and source term, with the chosen reference cell $Y:=[0,1]^2$ outlined on \cref{fig:bi-periodic-K}. Every mesoscopic grid will be square unless specified otherwise.

\begin{figure}
  \centering
  \begin{subfigure}{.48\linewidth}
    \centering{}
    \includegraphics[width=\linewidth]{missing-squares_K.tikz}
    \caption{Conductivity $K$}
    \label{fig:bi-periodic-K}
  \end{subfigure}
  ~
  \begin{subfigure}{.48\linewidth}
    \centering{}
    \includegraphics[width=\linewidth,height=\linewidth]{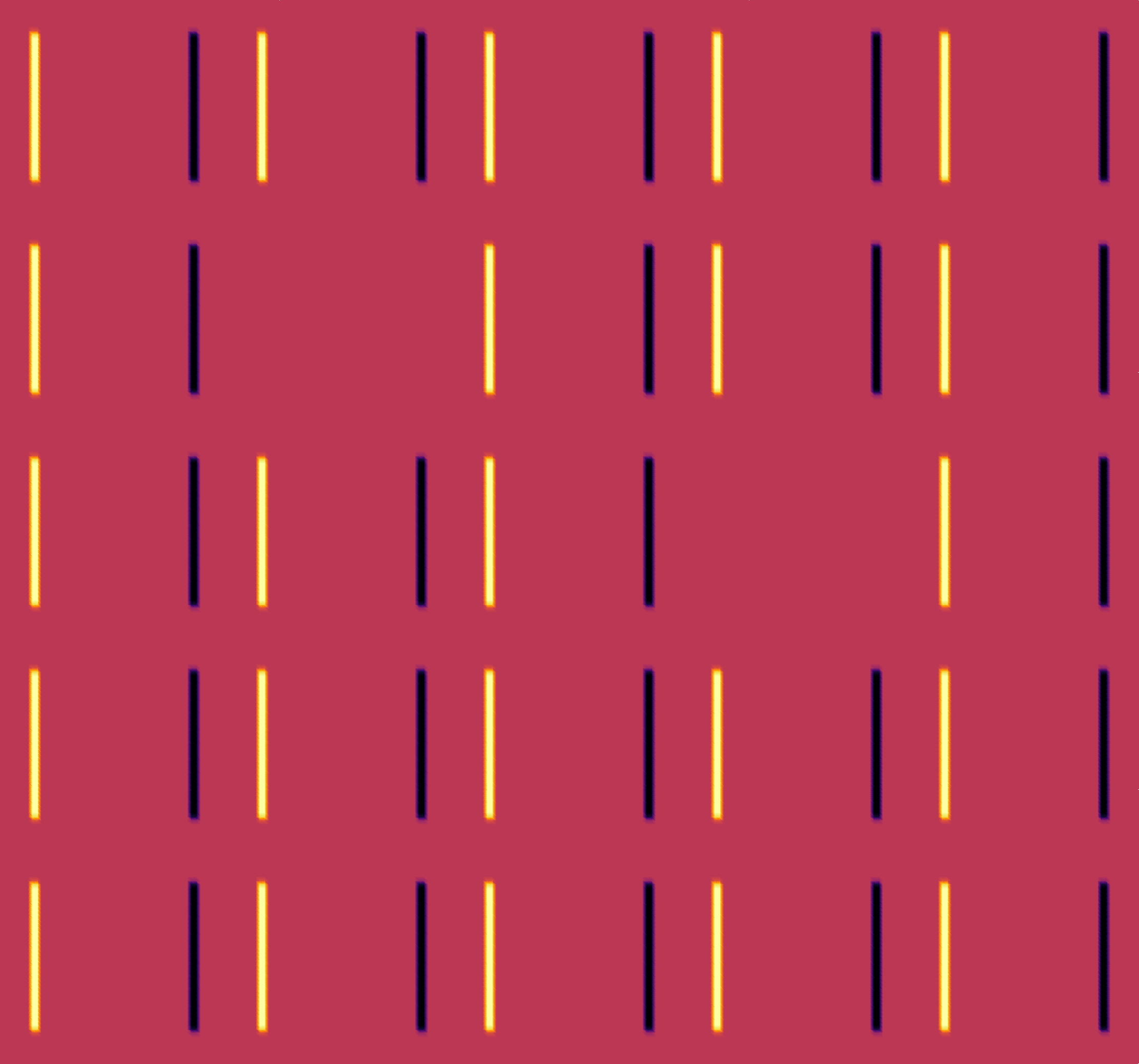}
    \caption{Source term $\diver Ke_1$}
    \label{fig:bi-periodic-src}
  \end{subfigure}
  
  \begin{subfigure}{.48\linewidth}
    \centering{}
    \includegraphics[width=\linewidth,height=\linewidth]{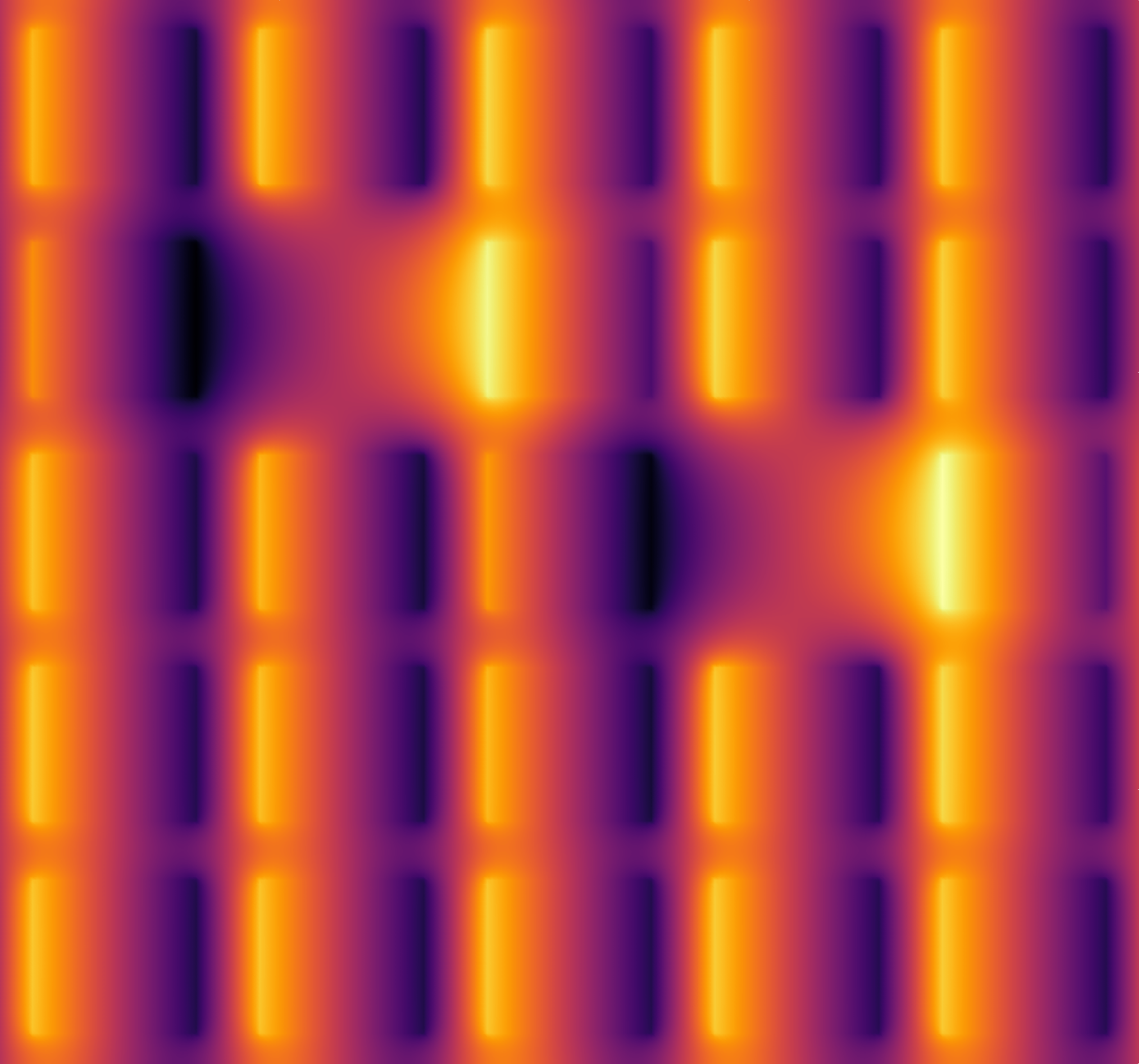}
    \caption{Corrector $w_{5,1}$}
    \label{fig:bi-periodic-corrector1}
  \end{subfigure}
  ~
  \begin{subfigure}{.48\linewidth}
    \centering{}
    \includegraphics[width=\linewidth,height=\linewidth]{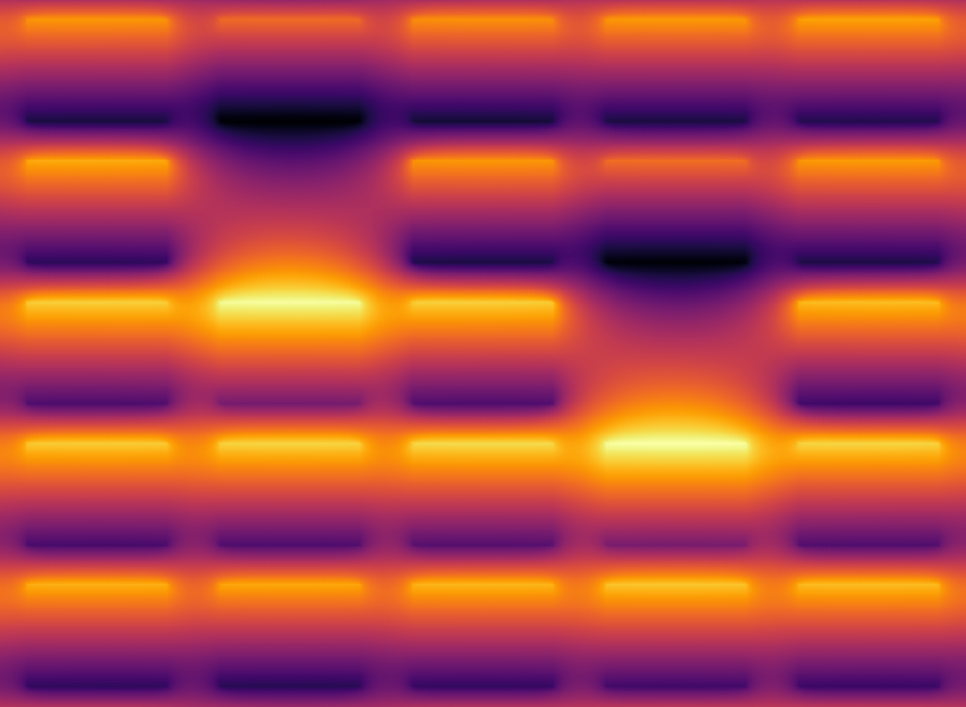}
    \caption{Corrector $w_{5,2}$}
    \label{fig:bi-periodic-corrector2}
  \end{subfigure}
  \caption{Main test case (example for $N^2 = 25$)}
  \label{fig:bi-periodic}
\end{figure}

Our reference when needed, both for accuracy and cost-efficiency evaluation, will be the exact same homogenisation process performed using a finite element method (FEM) with piecewise linear approximation for the corrector problems~\cref{eq:app-corr-sto}. For comparison purposes, its approximation space is built on an identical mesh, with the same degree of polynomials. This results in a problem of comparable size, if slightly smaller: a square domain of 400 cells with 400 finite elements each yield a FEM approximation space of dimension \num{160801} and a tensor approximation space $V_h(D)$ of dimension \num{176400}.

\label{sec:hardware}
All computations were run on the same workstation, \myie{} a Dell\texttrademark{} Optiplex\texttrademark{} 7010 with:
\begin{itemize}
\item \SI{8}{\gibi\byte}%
  \footnote{\num{5.5} to \num{6.5} of which were usually available for the simulations.}
  ($2\times4$) RAM DDR3 \SI{1600}{\mega\hertz};
\item Intel\textregistered{} Core\texttrademark{} i7-3770 CPU: 4 cores at \SI{3.40}{\giga\hertz} with 2 threads each.
\end{itemize}

\subsubsection{Approximation precision}
\label{sec:precision}

The trade-off between precision and cost-efficiency is fundamental to any approximation technique. We wish to assess the effect of MsLRM error on our quantity of interest $K_\star^{N}$. We computed the exact same%
\footnote{The same samples where used.}
Monte Carlo estimator $\theta_{20}(K_\star^{20})$ by FEM and MsLRM, with different tolerances $\varepsilon$ for MsLRM: \numlist{e-1;e-2;e-3;e-4;e-5}. FEM computational cost caused us to reduce the domain size and set an arbitrary low sample size, none of which interferes with the purpose of this test.

We can see on \cref{fig:precision_conv} how it affects the estimator's convergence, and it seems that a high tolerance shifts uniformly the estimation toward lower values: with inaccurately approximated correctors, the estimator underestimates the exact value (with a discrepancy apparently independent of $m$). Consequently, the error due to the approximation can not be compensated by a greater number of samples. We believe the cause to be an underestimation of the correctors' peaks at phase boundaries (visible on \cref{fig:bi-periodic-corrector1,fig:bi-periodic-corrector2} as darker or lighter areas) so that the approximated correctors represent a medium with smaller conductivity jumps at phase boundary, hence a lower homogenised conductivity. However, we observe no significant accuracy improvement on the homogenised quantity beyond a tolerance of \num{1e-3}: both the approximation of \numlist{1e-3;1e-5} seem as precise as the FEM reference. Even the approximation at \num{1e-2} yields very close results.

\Cref{fig:precision_time} shows the significant effect of $\varepsilon$ on cost-efficiency: it is clear that reducing the tolerance increases considerably the computational cost. There seems to be a tipping point between \numlist{1e-2;1e-3} after which the cost increase is steeper. The tolerance value will be set to \num{1e-2} henceforth as a satisfactory compromise.

\begin{figure}
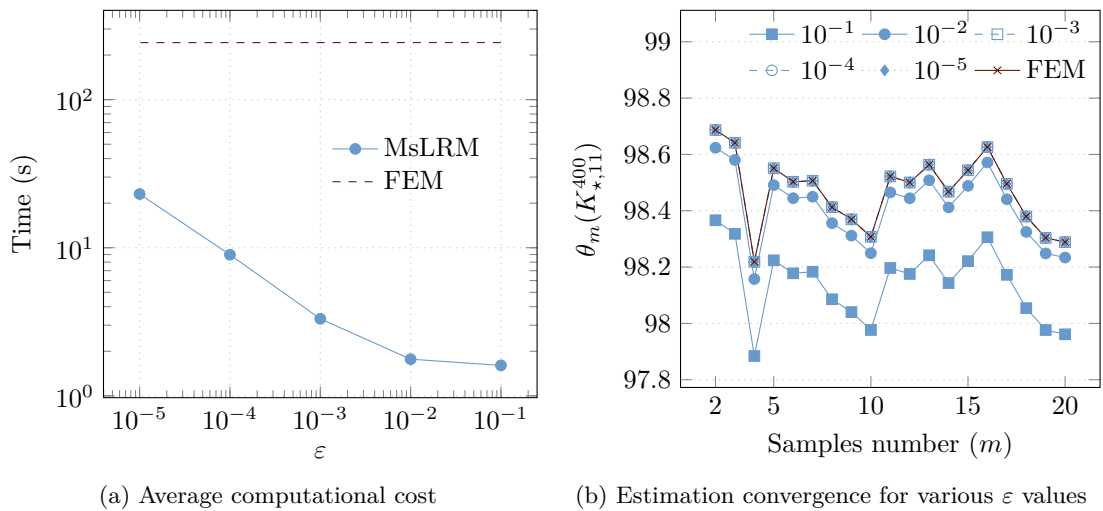

  \centering
  \begin{subfigure}{.48\linewidth}
    \centering
    \includegraphics[width=\linewidth,height=\plotDefaultHeight]{precision_time.tikz}
    \caption{Average computational cost}
    \label{fig:precision_time}
  \end{subfigure}
  ~
  \begin{subfigure}{.48\linewidth}
    \centering
    \includegraphics[width=\linewidth,height=\plotDefaultHeight]{precision_conv.tikz}
    \caption{Estimation convergence for various $\varepsilon$ values}
    \label{fig:precision_conv}
  \end{subfigure}
  \caption{MsLRM performance for various tolerances (\num{20} samples; $N^2=400$)}
  \label{fig:precision}
\end{figure}

\subsubsection{Modes recycling}
\label{sec:modes-recycling}

To illustrate modes recycling benefits, we computed the estimator $\theta^\eta(K^{N}_\star)$ with and without recycling. For the sake of graph clarity, only the first 20 samples were displayed on \cref{fig:recycling}. We see on \cref{fig:recycling_iter} the average corrector's rank per sample (averaged over both correctors). With recycling the average rank reaches immediately a plateau, which means the modes library was complete enough after the first sample. The computational cost is not reduced to zero however, as is shown on \cref{fig:recycling_time}. The larger the modes library, the higher the cost of initial step (\myie{} computing the best approximation possible on the current modes library), hence the drawback of non-optimal low-rank approximation.

Nevertheless, this reduces significantly the time required by subsequent computations, as can be seen on \cref{tab:recycling} from the average times and ranks, computed from all samples (not only the first 20). For comparison, the FEM time has been measured from a single computation. The average rank confirms that the modes library was complete with the first 6 modes.

\begin{figure}
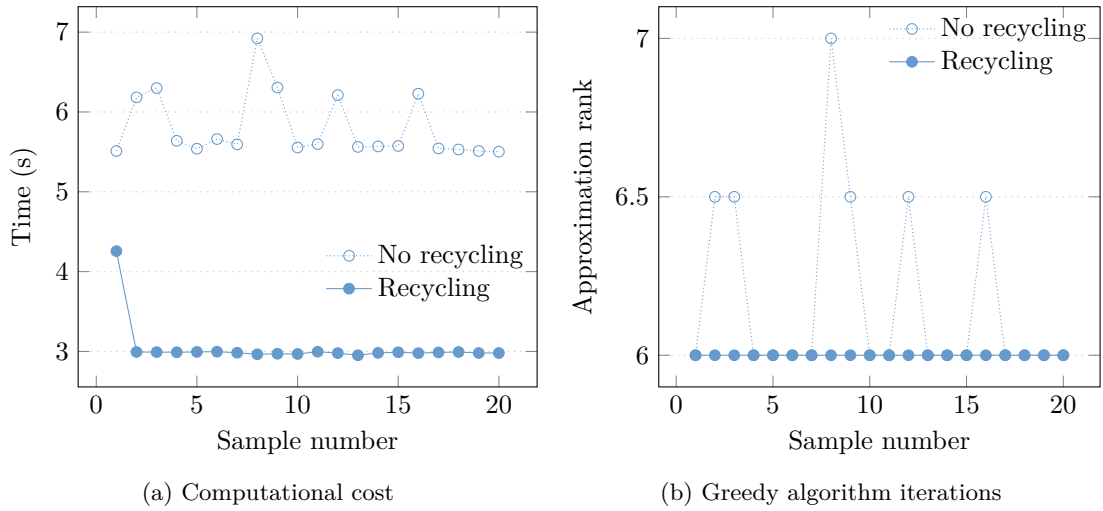

  \centering
  \begin{subfigure}{.48\linewidth}
    \centering
    \includegraphics[width=\linewidth,height=\plotDefaultHeight]{recycling_time.tikz}
    \caption{Computational cost}
    \label{fig:recycling_time}
  \end{subfigure}
  ~
  \begin{subfigure}{.48\linewidth}
    \centering
    \includegraphics[width=\linewidth,height=\plotDefaultHeight]{recycling_rank.tikz}
    \caption{Greedy algorithm iterations}
    \label{fig:recycling_iter}
  \end{subfigure}
  \caption[Modes recycling effect on MsLRM]{Modes recycling effect on MsLRM (averages across samples)}
  \label{fig:recycling}
\end{figure}

\begin{table}[hbt]
  \centering
  \caption{Comparison between FEM, MsLRM (with and without recycling)}
  \begin{tabular}{lcSS}\toprule
    & {FEM} & {MsLRM} & {MsLRM+recycling} \\\midrule
    {Average time per sample (s)}  & 4040  & 5.76  & 3               \\
    {Average rank per computation} & {--}  & 6.15  & 6               \\\bottomrule
  \end{tabular}
  \label{tab:recycling}
\end{table}

\subsubsection{Conductivity contrast}
\label{sec:contrast}

Here we evaluated the Monte Carlo estimator $\theta^\eta(K_\star^{N})$ with the same samples for different values of $K_2/K_1$: \numlist{2;e1;e2;e3;e4;e5}. The number of samples required were computed with formula~\cref{eq:mc-conv-rate} with a variance estimated with \num{100} samples. It can be seen on \cref{fig:contrast_conv} that this contrast affects greatly the estimator's variance, as expected. On the other hand, a high contrast seems to reduce the approximation cost (see \cref{fig:contrast_time}). One explanation is that the correctors' approximation is easier when the inter-phase peaks are so great that other variations become negligible. From now on, we will keep the conductivity contrast at \num{1e2} so that our medium remains significantly heterogeneous.

\begin{figure}
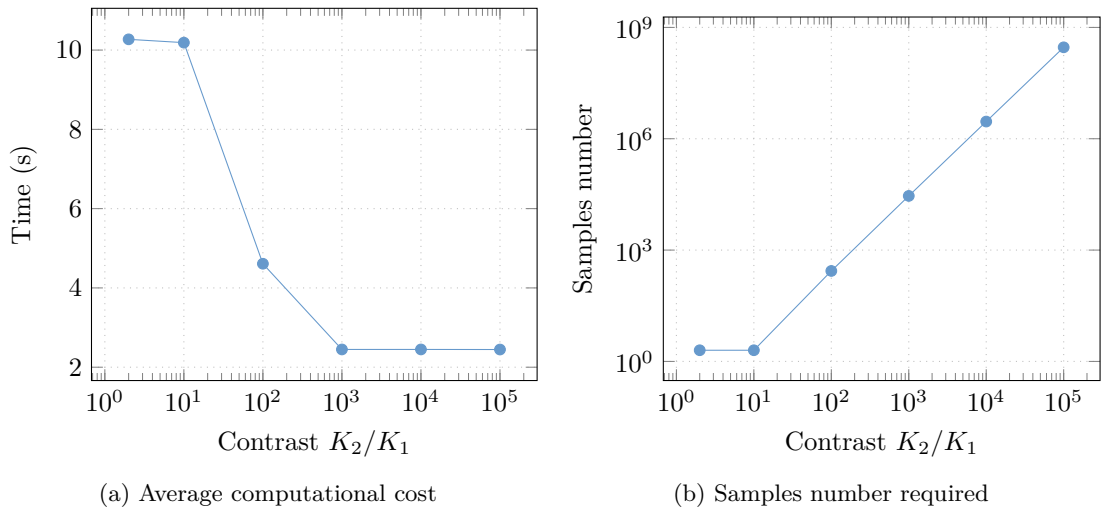

  \centering
  \begin{subfigure}{.48\linewidth}
    \centering
    \includegraphics[width=\linewidth,height=\plotDefaultHeight]{contrast_time.tikz}
    \caption{Average computational cost}
    \label{fig:contrast_time}
  \end{subfigure}
  ~
  \begin{subfigure}{.48\linewidth}
    \centering
    \includegraphics[width=\linewidth,height=\plotDefaultHeight]{contrast_conv.tikz}
    \caption{Samples number required}
    \label{fig:contrast_conv}
  \end{subfigure}
  \caption{Conductivity contrast effect on MsLRM performance}
  \label{fig:contrast}
\end{figure}

\subsubsection{Defect probability}
\label{sec:probability}

In a defect-type model such as described by equation~\cref{eq:K-bernoulli}, the defect probability $p$ is an important characteristic of the random medium. That is one way to quantify periodicity loss. We computed an estimation $\theta^\eta(K_\star^{N})$ for different values of $p$: \numlist{.01;.1;.2;.3;.4;.5}.

\Cref{fig:proba_rank} shows that the average approximation rank of correctors increases with $p$, as expected, until it reaches a plateau. Beyond a certain value of $p$ (\num{0.3} here), every configuration that would require new modes has become likely enough to occur, hence a rank plateau. We can relate this observation to the library completion previously discussed.

\Cref{fig:proba_conv} shows the impact of $p$ on the estimator's variance (recalling the estimator's convergence rate~\cref{eq:mc-conv-rate}). We set $p$ to \num{0.1} to keep the number of samples moderate.

\begin{figure}
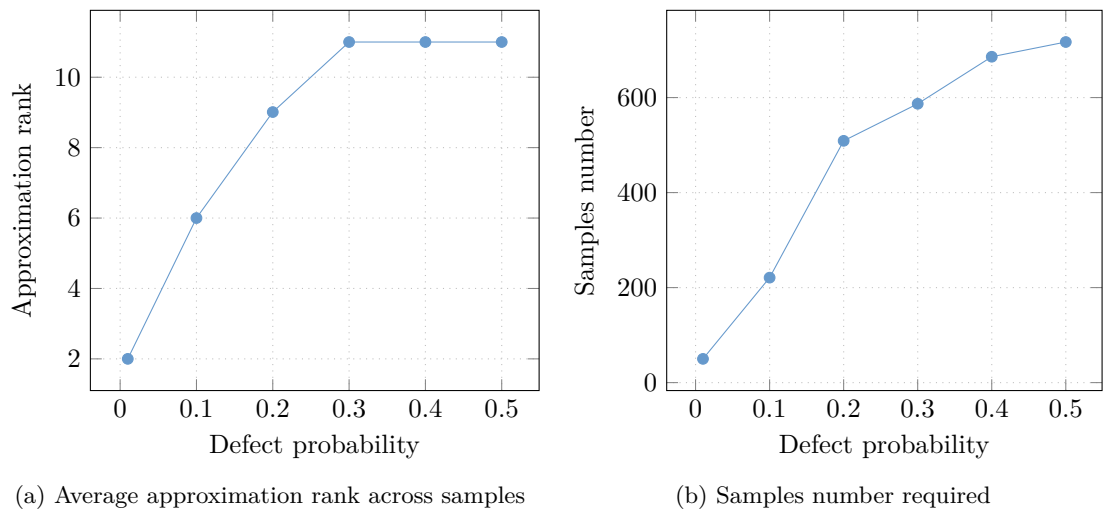

  \centering
  \begin{subfigure}{.48\linewidth}
    \centering
    \includegraphics[width=\linewidth,height=\plotDefaultHeight]{proba_rank.tikz}
    \caption{Average approximation rank across samples}
    \label{fig:proba_rank}
  \end{subfigure}
  ~
  \begin{subfigure}{.48\linewidth}
    \centering
    \includegraphics[width=\linewidth,height=\plotDefaultHeight]{proba_conv.tikz}
    \caption{Samples number required}
    \label{fig:proba_conv}
  \end{subfigure}
  \caption{Defect probability effect on MsLRM performance}
  \label{fig:proba}
\end{figure}

\subsubsection{Mesoscopic size}
\label{sec:mesoSize}

Recalling result~\cref{eq:mc-conv-rate}, building an estimator with a good convergence rate is a compromise between domain size (which affects variance, see \cref{sec:var-reduction}) and number of samples. We showed a way to curb MsLRM computational cost increase with respect to sample number with modes recycling. However, an immediate advantage of MsLRM is its low sensitivity to domain size increase, even though modes recycling further reduces it. To illustrate this, we evaluated $\theta^\eta(K_\star^{N})$ for several numbers of cells: \numlist{100;400;900;1600;2500;3600}.

Additionally, we wish to broach the case of unidirectional periodicity. We dealt only with bi-directional periodicity so far, on square mesoscopic grids. Let us consider a unidirectional mesoscopic grid with reference cell $Y:=[0,1]\times[0,10]$ so that a five-cells domain is $[0,5]\times[0,10]$ as illustrated on \cref{fig:mono-periodic} (with $Y$ outlined on \cref{fig:mono-periodic-K}). We choose a conductivity $\hat{K}$ of the same form~\cref{eq:K-bernoulli} but with $\chi$ replaced with $\hat{\chi}$ to get the patterns shown on \cref{fig:mono-periodic-K}. This test case was inspired by unidirectional fibre-reinforced composite materials. We evaluate $\theta^\eta(\hat{K}_\star^{N})$ as well, with the same values as above for the number of cells and the other parameters still set according to \cref{tab:default-param}.

\begin{figure}
  \centering
  \begin{subfigure}{.3\linewidth}
    \centering{}
    \includegraphics[height=\pictureDefaultHeight]{missing-fibres_K.tikz}
    \caption{Conductivity $K$}
    \label{fig:mono-periodic-K}
  \end{subfigure}
  ~
  \begin{subfigure}{.3\linewidth}
    \centering{}
    \includegraphics[height=\pictureDefaultHeight, width=.5\pictureDefaultHeight]{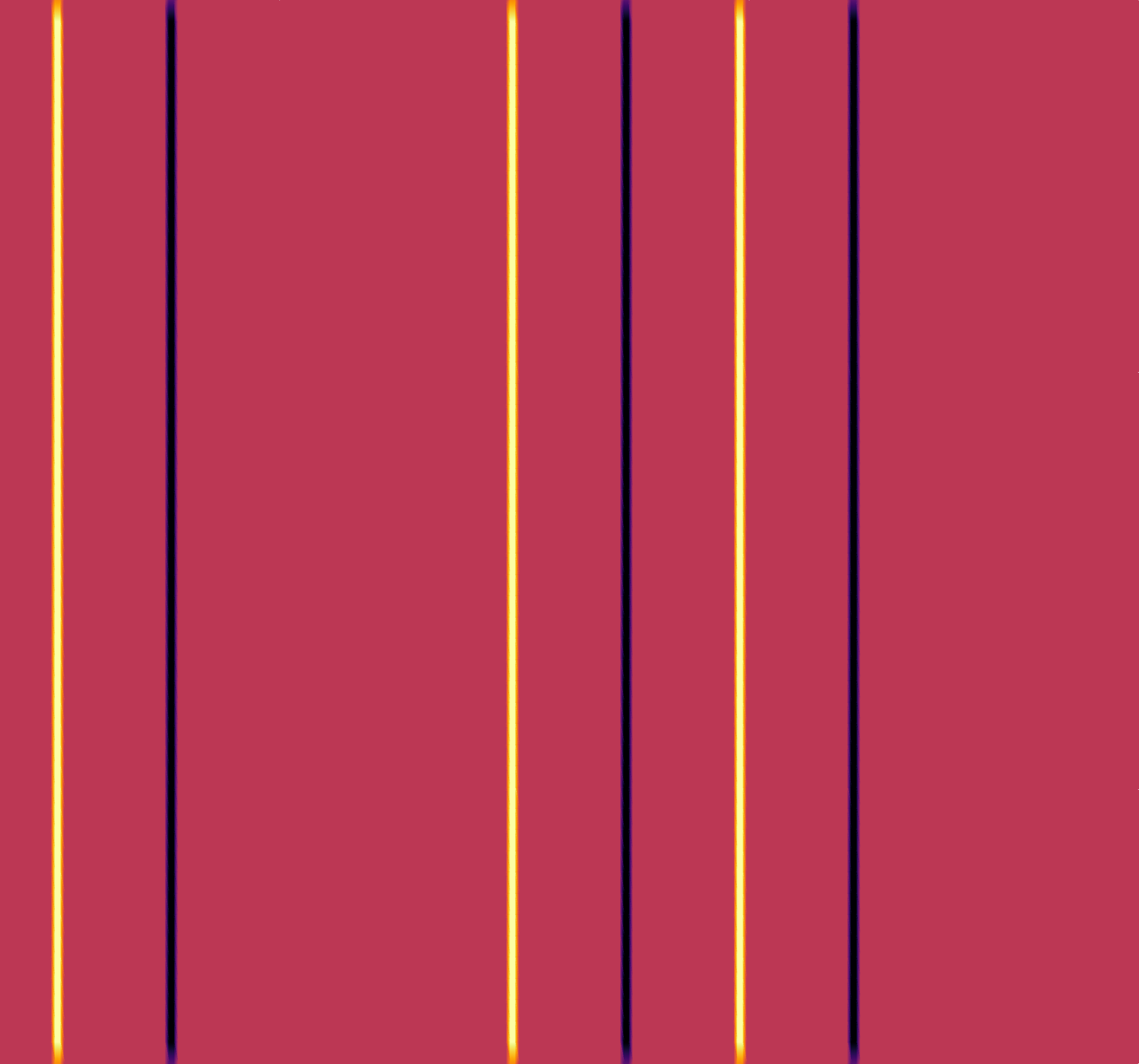}
    \caption{Source term $f=\diver Ke_1$}
  \end{subfigure}
  ~
  \begin{subfigure}{.3\linewidth}
    \centering{}
    \includegraphics[height=\pictureDefaultHeight, width=.5\pictureDefaultHeight]{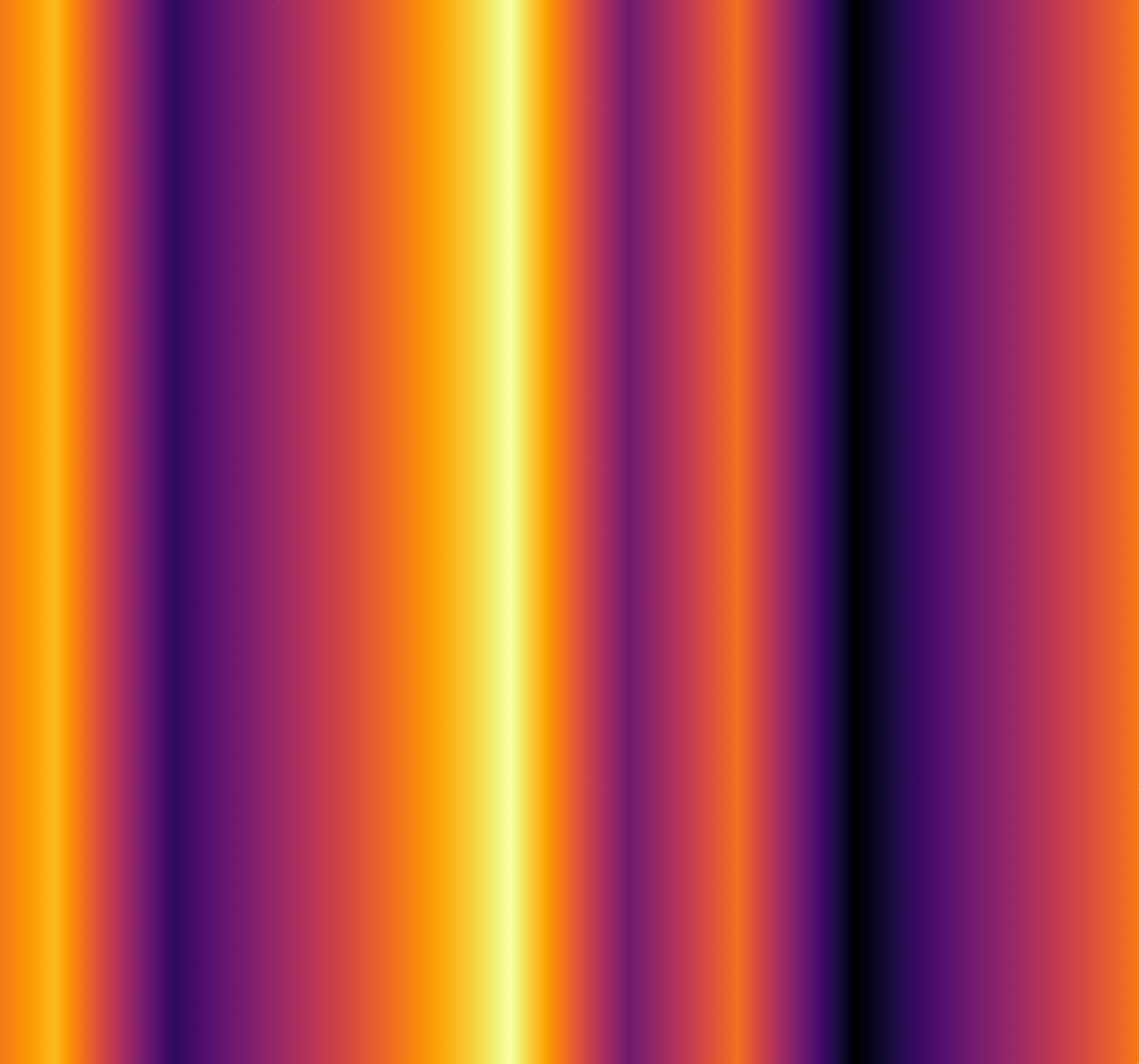}
    \caption{Corrector $w_1$}
  \end{subfigure}
  \caption{Unidirectional periodic medium (example with five cells)}
  \label{fig:mono-periodic}
\end{figure}

First, we see on \cref{fig:mesoSize_time} the complexity reduction achieved, whatever the mesoscopic dimension. The MsLRM computational cost increase when the number of cells grows is moderate, particularly in comparison to the cost of the corresponding FEM solution%
\footnote{Measured from a single computation.}
(displayed on the same graph, with identical values for both cases). Furthermore, \cref{fig:mesoSize_rank} shows the rank's plateau caused by $K$'s stationarity, as was discussed in \cref{sec:lrm} on the topic of modes recycling.

Secondly, it is clear from \cref{fig:mesoSize_rank} that the complexity of the unidirectional case is considerably lower. A uni-dimensional grid brings the approximation's rank much closer to $K$'s: any local loss of periodicity (faulty cell here) has much fewer neighbours to affect than in the previous case. The slight difference on \cref{fig:mesoSize_time} comes from a better conditioning of the bi-directional problem, due to homogeneity of $K$ on cells boundaries.

As for the approximation in the bi-directional case, although it is almost constant the first point (\num{100} cells) on \cref{fig:mesoSize_rank} is worth commenting. We mentioned that the modes library is built in a non-optimal way; furthermore, it can only grow. Here, for every domain of \num{400} cells or more, the completion was almost immediate, \myie{} the first sample was a good representation of the medium as far as the algorithm was concerned. With \num{100} cells, however, several samples had to be processed before reaching completion, each one causing some modes to be added to the library. These modes were the best choices for the sample, but not necessarily with regards to every possible realisation, hence a larger library.

Finally, \cref{tab:mesoSize} shows how domain size affects the convergence rate of Monte Carlo estimation, as was mentioned above. Those values were estimated from \num{100} samples (as were all averages computed for this test case). For these estimators of identical variance $\eta$, the improved convergence rate somewhat balances the increase in computational time per sample in the overall cost of estimation, as the measured total times testify. Consequently, and although there is generally a compromise to be found to minimise the cost (more details in \cref{sec:var-reduction}), we observe a wide range of acceptable values for $\#I$. The cost seems minimal around \num{1600} cells, so we chose this value as default (see \cref{tab:default-param}).

\begin{figure}
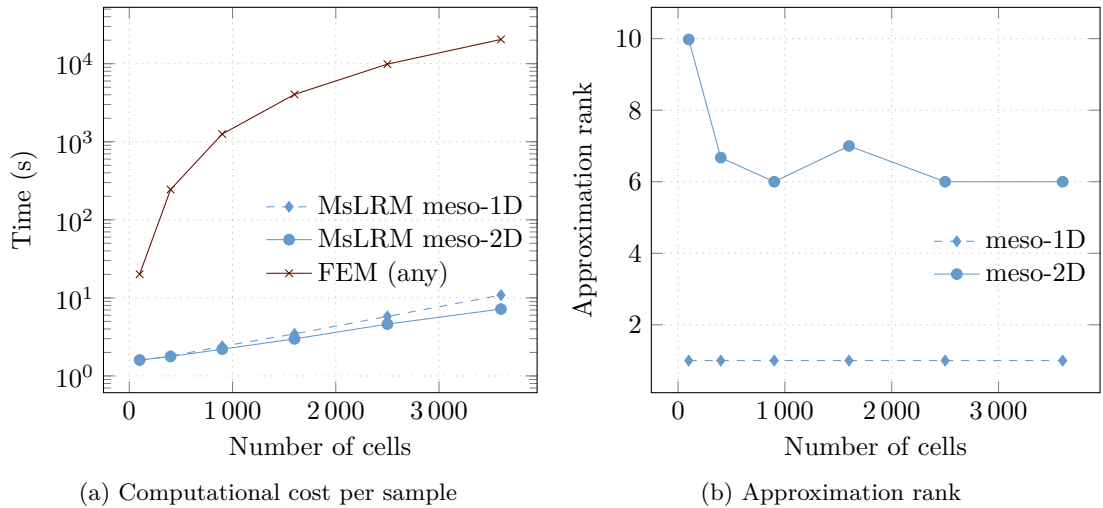

  \centering
  \begin{subfigure}{.48\linewidth}
    \centering
    \includegraphics[width=\linewidth,height=\plotDefaultHeight]{mesoSize_time.tikz}
    \caption{Computational cost per sample}
    \label{fig:mesoSize_time}
  \end{subfigure}
  ~
  \begin{subfigure}{.48\linewidth}
    \centering
    \includegraphics[width=\linewidth,height=\plotDefaultHeight]{mesoSize_rank.tikz}
    \caption{Approximation rank}
    \label{fig:mesoSize_rank}
  \end{subfigure}
  \caption{Domain size effect on MsLRM and FEM performance for uni- and bi-directional periodicity (averages across samples)}
  \label{fig:mesoSize}
\end{figure}

\begin{table}
  \centering
  \caption{Domain size effect on samples number $m(\eta)$ and time required for estimation}
    \begin{tabular}{@{}c*{8}{S[table-format=4]}S[table-format=3]*{3}{S[table-format=4]}@{}}\toprule
        & \multicolumn{6}{c}{1D} & \multicolumn{6}{c}{2D}                                                    \\\cmidrule(lr){2-7}\cmidrule(l){8-13}
    $\#I$               & 100                    & 400  & 900  & 1600 & 2500 & 3600 & 100  & 400  & 900 & 1600 & 2500 & 3600 \\\midrule
    $m(\eta)$ & 4116                   & 1107  & 467  & 229  & 164  & 106  & 4015 & 1089  & 457 & 227  & 163  & 105  \\ 
    Time (s)            & 6480                   & 2012 & 1164 & 813 & 1030 & 1142 & 6552 & 1927 & 979 & 746  & 758  & 761  \\\bottomrule
  \end{tabular}
  \label{tab:mesoSize}
\end{table}

\subsubsection{Aperiodic media with mesoscopic structure}
\label{sec:aperiodic}

In order to address a more general case where both the conductivity $K_2$ and the area it encompasses (represented by the characteristic function $\chi$) may be random, we consider a broader variant of conductivity expression~\cref{eq:K-bernoulli}. For $(i,y,\omega)\in I\times Y\times\Omega$, we let
\begin{gather*}
  K(i,y,\omega) = K_1 + \chi_i(y,\omega)(K_{2,i}(\omega) - K_1),
\end{gather*}
where $K_{2,i}(\omega)\in \mathcal{U}([K_2^-,K_2^+])$ are independent and identically distributed, and each $\chi_i(\cdot,\omega):Y\to\{0,1\}$ is the characteristic function of the square of side length $\ell_i(\omega)$ centred relatively to the cell. The random variables $\ell_i\in\mathcal{U}(\Zint{L_1}{L_q}), i\in I,$ are \myiid{} as well. An example is displayed on \cref{fig:aperiodic-example} and the chosen values for our tests are in \cref{tab:uniform-laws}.

\begin{table}
  \centering
  \caption{Uniform laws' parameters}
  \begin{tabular}{>{$}c<{$}>{$}c<{$}>{$}c<{$}}\toprule
    K_2^- & K_2^+ & (L_i)_{i\in\Zint{1}{10}} \\\midrule
    1     & 100   & (i-1)\times 0.1          \\\bottomrule
  \end{tabular}
  \label{tab:uniform-laws}
\end{table}

We call such medium \enquote{aperiodic} because, for $i\neq j$, $\mathbb{P}(K(i,\cdot,\cdot)=K(j,\cdot,\cdot)) = 0$. Nevertheless, there is a spatial structure such that a mesoscopic grid can be defined, and we outlined the chosen reference cell $Y$ on \cref{fig:aperiodic-K}. Here the rank $r_K$ of $K$ is a random variable whose values are bounded according to $(\ell_i)_{i\in I}$ probability laws: $r_K\leqslant 1+q$ almost surely in this case. However, had we chosen $(\ell_i)_{i\in I}$ to be \emph{continuous} random variables (\myeg{} $\ell_i\in\mathcal{U}([L_1,L_q])$), then $r_K=N^2$ almost surely.

\begin{figure}
  \centering
  \begin{subfigure}{.48\linewidth}
    \centering{}
    \includegraphics[width=\linewidth,height=\linewidth]{aperiodic_conductivity.tikz}
    \caption{Conductivity $K$}
    \label{fig:aperiodic-K}
  \end{subfigure}
  ~
  \begin{subfigure}{.48\linewidth}
    \centering{}
    \includegraphics[width=\linewidth,height=\linewidth]{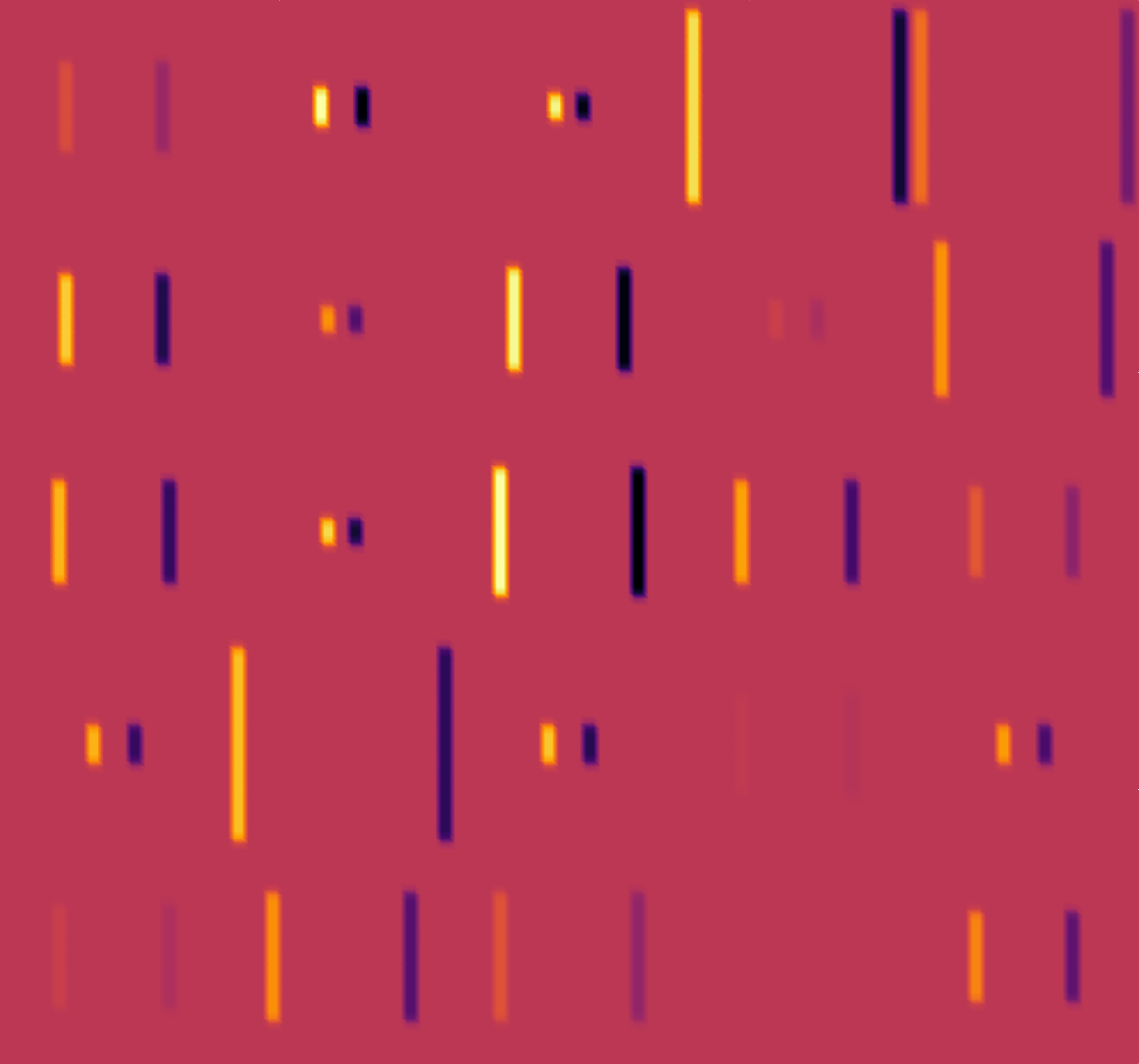}
    \caption{Source term $\diver Ke_1$}
    \label{fig:aperiodic-src}
  \end{subfigure}
  
  \begin{subfigure}{.48\linewidth}
    \centering{}
    \includegraphics[width=\linewidth,height=\linewidth]{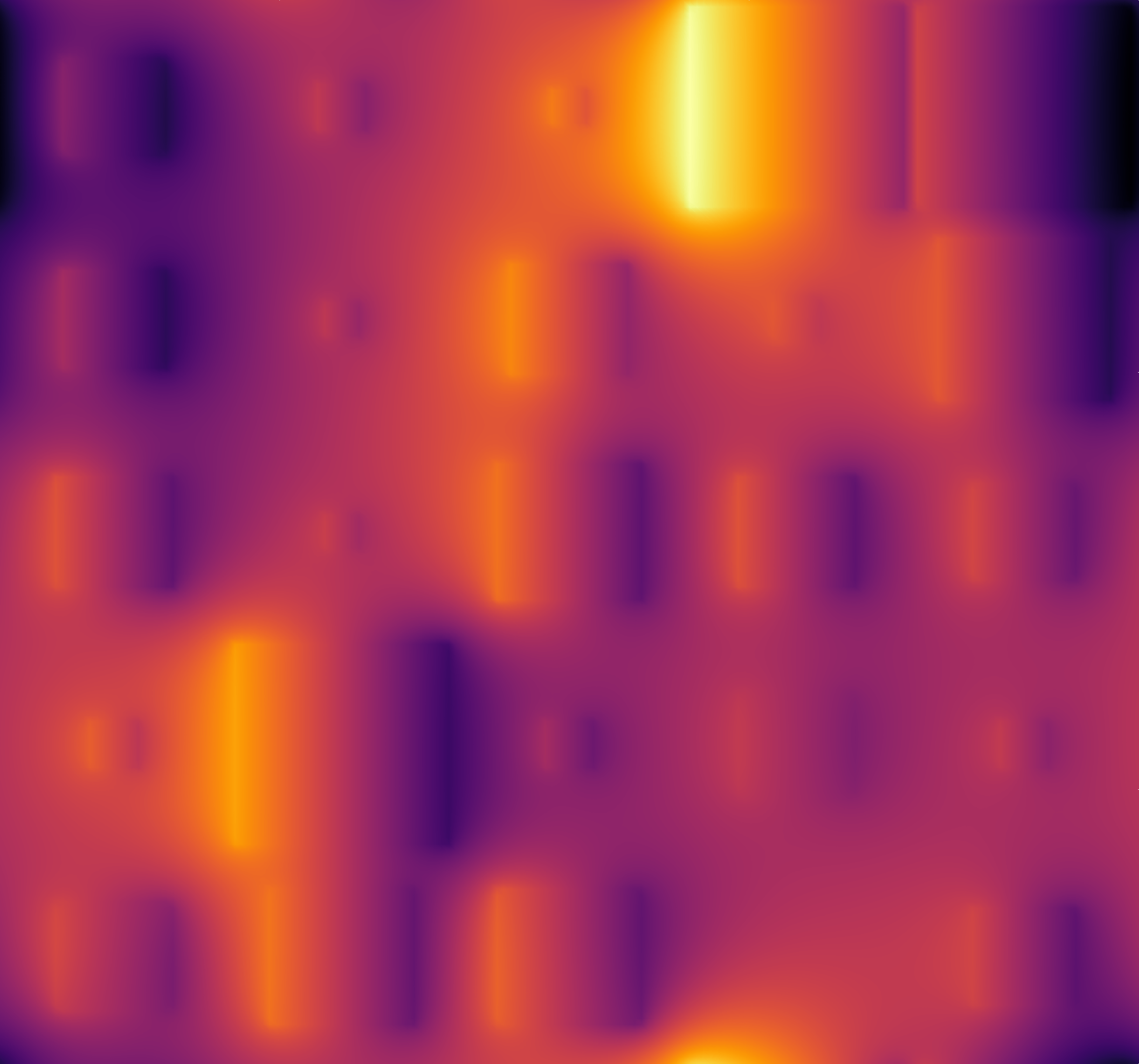}
    \caption{Corrector $w_{5,1}$}
    \label{fig:aperiodic-corrector1}
  \end{subfigure}
  ~
  \begin{subfigure}{.48\linewidth}
    \centering{}
    \includegraphics[width=\linewidth,height=\linewidth]{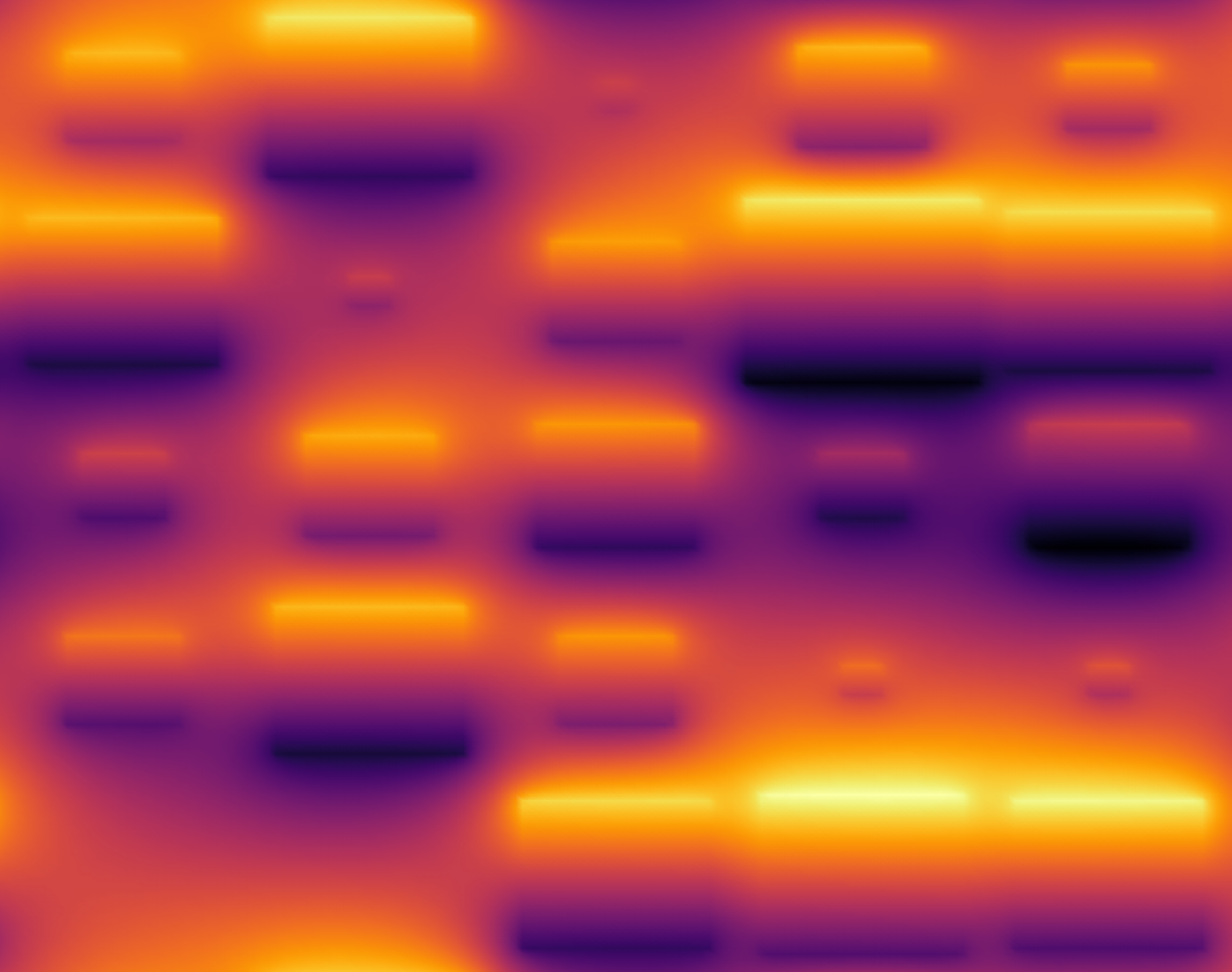}
    \caption{Corrector $w_{5,2}$}
    \label{fig:aperiodic-corrector2}
  \end{subfigure}
  \caption{Aperiodic medium example for \num{25} cells}
  \label{fig:aperiodic-example}
\end{figure}

To get insight into the complexity of such problem, we perform the same tests as in \cref{sec:mesoSize} (only the bi-directional case). The comparison of \cref{fig:aperiodic_rank,fig:mesoSize_rank} reveals the increased complexity of the current case, and explains the difference between \cref{fig:aperiodic_time,fig:mesoSize_time}; this was expected from the conductivities' rank difference: 2 against 10. Incidentally, this difference is due exclusively to the geometric variability introduced with $(\chi_i)$; indeed a rank-2 \enquote{aperiodic} conductivity (in the sense explained above) is possible. However, the FEM computation times remind that the complexity reduction remains significant even in such aperiodic medium. As a side note, the computer reached its memory limit during MsLRM computations for $N^2:=3600$, which somewhat diminished the performance.

Additionally, the increased rank of the approximations makes modes recycling particularly relevant here: the average values on \cref{fig:aperiodic-results} are based on \num{100} samples, and a greater number of samples would decrease the average time as the impact of the first computations' cost (from which the library was built) fades. To give more insight, the cost per sample on a complete library varied between \SIlist{5;17}{\percent} of the first sample's computation cost. Modes library was essentially completed within the first \num{8} samples, except the case $N^2=100$ which required around \num{35}. The decreasing trend on \cref{fig:aperiodic_rank} has the same explanation as was given in \cref{sec:mesoSize}, with greater variability in $K$.

Finally, we show on \cref{tab:aperiodic} the number of samples required and associated total computation time (estimated from \num{100} samples). The comparison with \cref{tab:mesoSize} shows how the increased computational cost has shifted the optimal cell number downward. In both tables, however, appears a minimal number of cells above which the total cost is of the same order of magnitude.

\begin{figure}
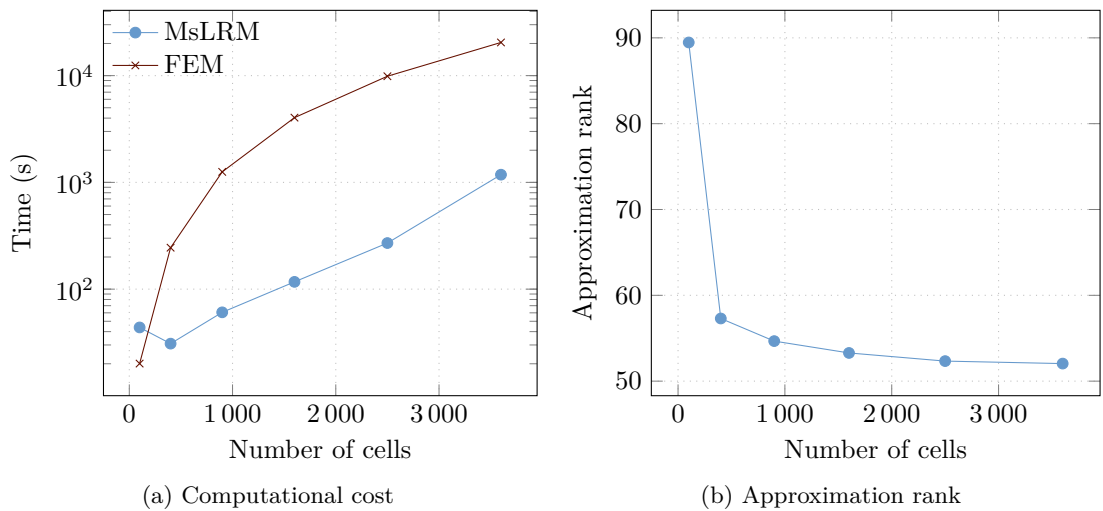

  \centering
  \begin{subfigure}{.48\linewidth}
    \centering
    \includegraphics[width=\linewidth,height=\plotDefaultHeight]{aperiodic_time}
    \caption{Computational cost}
    \label{fig:aperiodic_time}
  \end{subfigure}
  ~
  \begin{subfigure}{.48\linewidth}
    \centering
    \includegraphics[width=\linewidth,height=\plotDefaultHeight]{aperiodic_rank}
    \caption{Approximation rank}
    \label{fig:aperiodic_rank}
  \end{subfigure}
  \caption{MsLRM performance on aperiodic medium (averages across samples)}
  \label{fig:aperiodic-results}
\end{figure}

\begin{table}
  \centering
  \caption{Sample number and time required on aperiodic medium}
  \begin{tabular}{c*{6}{S[table-format=6]}}\toprule
    $N^2$   & 100  & 400 & 900 & 1600 & 2500 & 3600 \\\midrule
    $m(\eta)$ & 4627 & 961 & 510 & 297  & 191  & 124  \\
    Time (s) & 202530 & 29637 & 30889 & 34707 & 51475 & 145960 \\\bottomrule
  \end{tabular}
  \label{tab:aperiodic}
\end{table}

\subsubsection{Transformed media}
\label{sec:mapping}

We illustrate here the use of MsLRM in the variant of stochastic homogenisation detailed in \cref{sec:homog-map}. We define a periodic $K_\sharp=K_1 + \hat{\chi}(K_2-K_1)$ with $\hat{\chi}$ introduced in \cref{sec:mesoSize} for the unidirectional case (see \cref{fig:mono-periodic}) and choose a random diffeomorphism $\phi$ and let $K:=K_\sharp\circ\phi^{-1}$. No fibre is missing this time, but the spaces between fibres are independent random variables $(\ell_i)_{i\in\Zint{0}{\#I}}$ following the same uniform law $\mathcal{U}([0.1,2])$. The two limit fibres in $D_1$ and $D_{\#I}$ are separated from the domain boundary by respectively $\ell_0/2$ and $\ell_{\#I}/2$ and the fibre width is a constant denoted $\delta$.  The chosen random diffeomorphism $\phi$ is then continuous, piecewise affine and expressed component-wise $\phi(i,y) = \phi_1(i,y)e_1 + (\zeta(i,y)\cdot e_2)e_2$. Therefore, nothing is changed along $e_2$ and $\phi_1$ has the rank-4 representation $\phi_1 = \phi_0^I\otimes 1^Y + \phi_1^I\otimes\phi^Y_1 + 1^I\otimes\phi^Y_2 + \phi_3^I\otimes\phi^Y_3$, where
\begin{gather*}
  \phi_0^I(i)=
  \begin{dcases}
    0 & \text{if } i=1 \\
    \frac{\ell_0}{2} + \frac{\ell_{i-1}}{2} + (i-1)\delta + \sum_{k=1}^{i-2}\ell_k & \text{else}
  \end{dcases},\;
  \phi_1^I(i)= \frac{\ell_{i-1}}{1-\delta}  ,\;
  \phi_3^I(i)= \frac{\ell_{i}}{1-\delta} ,
\end{gather*}
and
\begin{multline*}
  \phi_1^Y(y)= \min\mleft(y_1,\frac{1-\delta}{2}\mright) ,\;
  \phi_2^Y(y)= \max\mleft(0,\min\mleft(\delta,y_1-\frac{1-\delta}{2}\mright)\mright) ,\; \\
  \phi_3^Y(y)= \max\mleft(0,y_1-\frac{1+\delta}{2}\mright).
\end{multline*}
We noted $1^I$ and $1^Y$ the constant functions of value \num{1} and $y_1:=y\cdot e_1$.
This gradient of this transformation is expressed
\begin{gather*}
  \grad\phi =
  \begin{pmatrix}
    \phi_{1,1} & 0 \\ 0 & 1^I\otimes 1^Y
  \end{pmatrix}
  \quad \text{with } \phi_{n,1} = \pderiv{\phi_n}{x_1},
\end{gather*}
so $\deter{\grad \phi} = \phi_{1,1} = \phi_1^I \otimes \phi^Y_{1,1} + 1^I\otimes\phi^Y_{2,1} + \phi_3^I\otimes\phi^Y_{3,1}$. Here $\grad\mleft(\phi^{-1}\mright)=\mleft(\grad\phi\mright)^{-1}$ thus the ensuing simplifications in $\mapK := \deter{\grad\phi} \tensor*[]{\grad\phi}{^{-1}} K_\sharp \tensor[^t]{\grad\phi}{^{-1}}$ yield $\rank{\mapK}=3$. The effects of this transformation are illustrated on \cref{fig:mapping-example} with an example of realisation of $K$ and its associated source term and first corrector function.

In this test, we solve the corrector problems~\cref{eq:app-corr-map}. We perform the same tests as for the unidirectional case in \cref{sec:mesoSize}. The performance, shown on \cref{fig:mapping-results}, is between that on \cref{fig:mesoSize} and that on \cref{fig:aperiodic-results}, which suggests that these tests are comparable enough for the performance to be dominated by $r_K$'s influence. This is a consequence of the correctors problems' nature: the source term follows the same quasi-periodicity as $K$, and boundary conditions do not break this quasi-periodicity in the solution. Therefore, the only periodicity loss in the solution comes from $K$.

\begin{figure}
  \centering{}
  \begin{subfigure}{.3\linewidth}
    \centering
    \includegraphics[width=.653\pictureDefaultHeight]{mapping-K.tikz}
    \caption{Conductivity $K$}
  \end{subfigure}
  ~
  \begin{subfigure}{.3\linewidth}
    \centering
    \includegraphics[width=.653\pictureDefaultHeight, height=\pictureDefaultHeight]{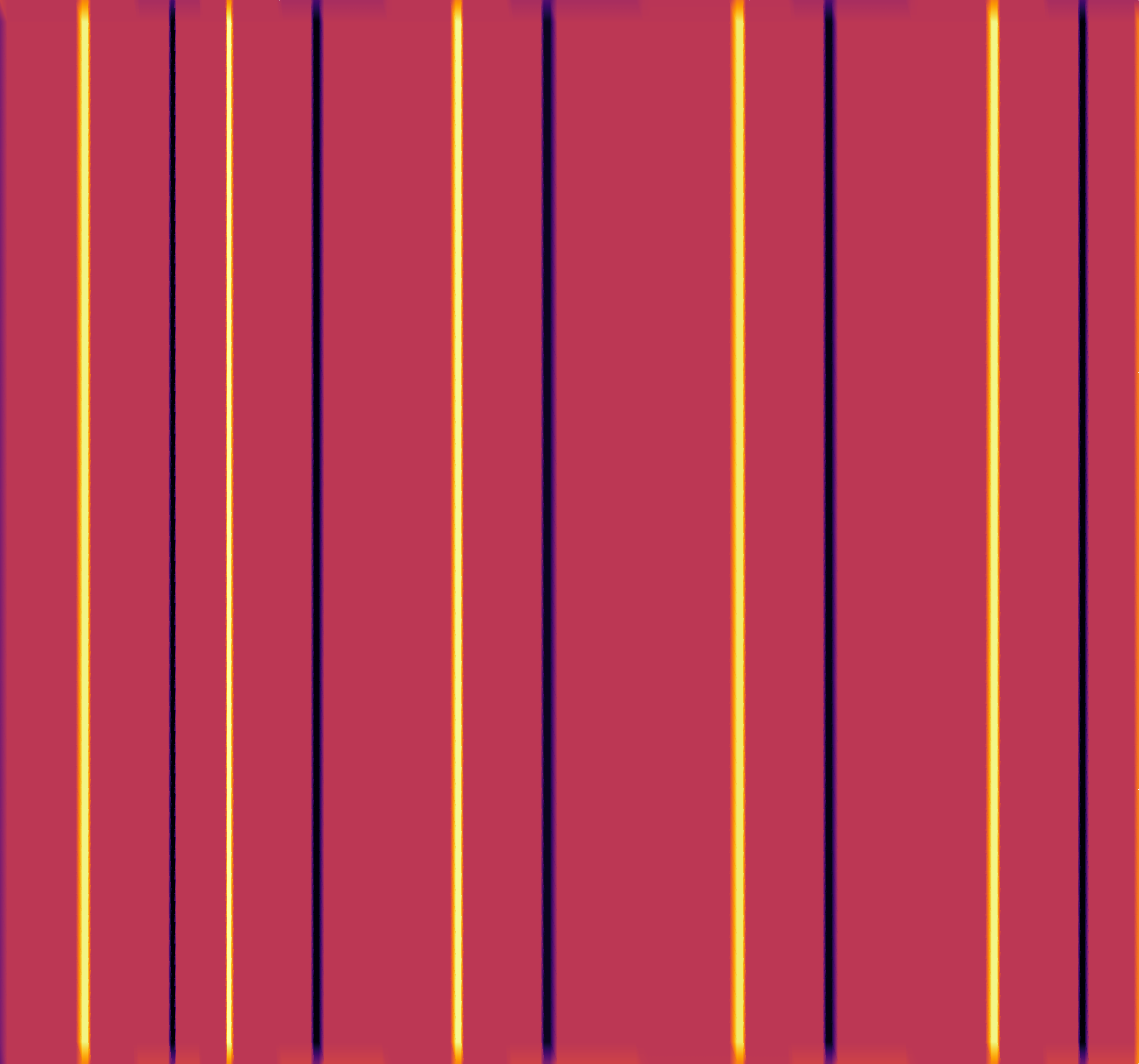}
    \caption{Source term $f=\diver(K e_1)$}
  \end{subfigure}
  ~
  \begin{subfigure}{.3\linewidth}
    \centering
    \includegraphics[width=.653\pictureDefaultHeight, height=\pictureDefaultHeight]{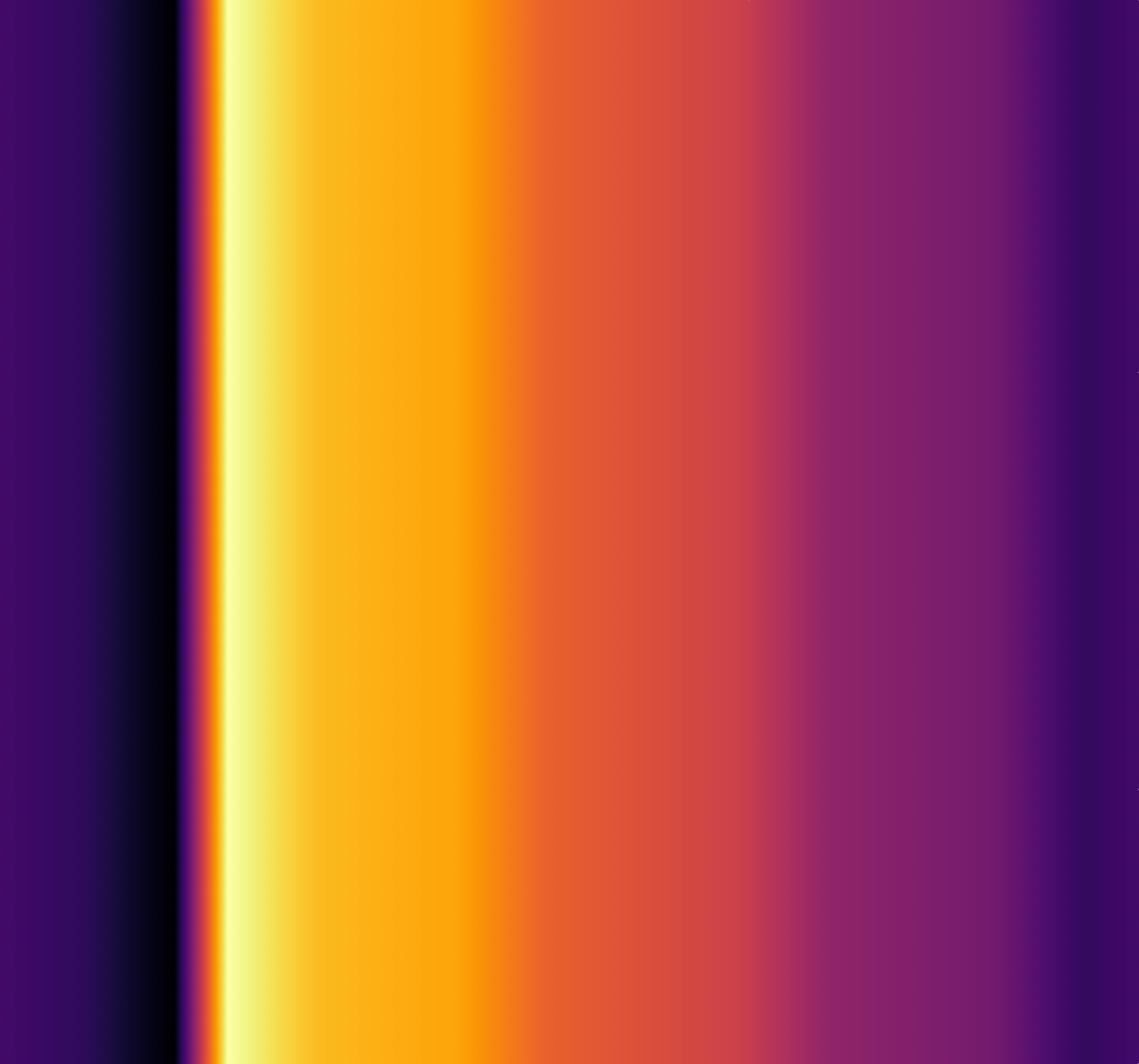}
    \caption{Corrector $w_{5,1}$}
  \end{subfigure}
  \caption{Example of randomly transformed medium for \num{5} cells.}
  \label{fig:mapping-example}
\end{figure}

\begin{figure}
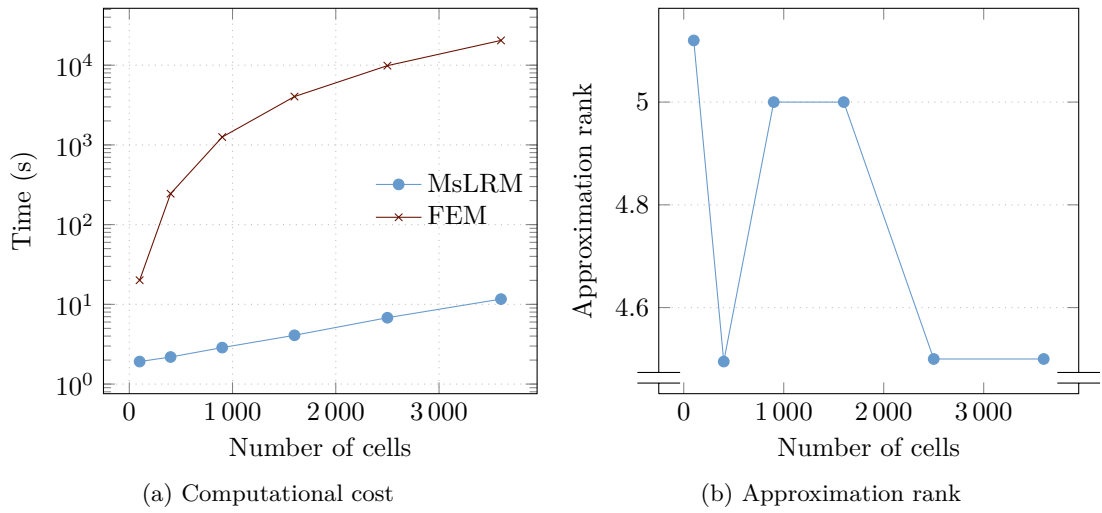

  \centering
  \begin{subfigure}{.48\linewidth}
    \centering
    \includegraphics[width=\linewidth,height=\plotDefaultHeight]{mapping_time.tikz}
    \caption{Computational cost}
    \label{fig:mapping_time}
  \end{subfigure}
  ~
  \begin{subfigure}{.48\linewidth}
    \centering
    \includegraphics[width=\linewidth,height=\plotDefaultHeight]{mapping_rank.tikz}
    \caption{Approximation rank}
    \label{fig:mapping_rank}
  \end{subfigure}
  \caption{MsLRM performance on transformed medium (averages across samples)}
  \label{fig:mapping-results}
\end{figure}

\section{Variance reduction by control variate}
\label{sec:var-reduction}

The error associated with approximation of $K_\star$ by $\theta_m(K_\star^N)$ (from~\cref{eq:mc-estimator})  can be decomposed as
\begin{gather*}
  K_\star - \theta_m(K_\star^N)
    = \underbracket[.8pt]{K_\star - \expec*{K_\star^N}}_{\text{\textbf{systematic} error}} + \underbracket[.8pt]{\expec*{K_\star^N} -  \theta_m(K_\star^N)}_{\text{\textbf{statistical} error}}.
  \end{gather*}
It has been proven that
\begin{gather*}
  \expec*{K_\star - \expec*{K_\star^N}}\lesssim N^{-d}\ln^d(N)
  \quad\text{\parencite[in][]{Gloria2011}} \\
  \text{and}\quad
  \sqrt{\expec*{(\expec*{K_\star^N}-\theta_m(K_\star^N))^2}} \lesssim (N^dm)^{-1/2}
  \quad\text{\parencites(in)(){Gloria2012}{Gloria2015a}}
\end{gather*}
for a cubic domain of side length proportional to $N$. Consequently, the bias decreases fast with domain size and the statistical error is usually the major contribution to the error. As was mentioned in equation~\cref{eq:mc-conv-rate}, this statistical error can be quantified as
\begin{gather*}
  \expec*{\mleft(\theta_m(\rvi{})-\expec*{\rvi{}}\mright)^2} = \var*{\theta_m(\rvi{})} = \frac{\var*{\rvi{}}}{m}
\end{gather*}
for any random variable $\rvi{}\in L^2(\Omega)$, hence the numerous variance reduction techniques developed to accelerate Monte Carlo method's convergence rate.
Such techniques include stratified sampling, importance sampling, antithetic variates, control variates and many more; we are interested in the latter.

Let $\rho\in\setR$ and let $Z\in L^2(\Omega)$ be a random variable. $Z$ serves as a control variate by defining the controlled variable $X = \rvi{} - \rho(Z-\expec{Z})$.
We have, purposely, $\expec{\rvi{}}=\expec{X}$.
As for the variance, $\var{X} = \var{\rvi{}} - 2\rho\covar{\rvi{},Z} + \rho^2\var{Z}$ and is minimal for
\begin{align}
  \label{eq:argmin-var}
  \rho & = \frac{\covar*{\rvi{},Z}}{\var*{Z}}, \\
  \intertext{with minimum}
  \label{eq:min-var}
  \min_{\rho\in\setR}\var*{X} & = \mleft(1-\frac{\mleft(\covar*{\rvi{},Z}\mright)^2}{\var*{\rvi{}}\var*{Z}}\mright)\var*{\rvi{}}.
\end{align}
We observe from equation~\cref{eq:min-var} that, whatever $Z$, a variance reduction can be achieved with a suitable parameter $\rho$. A good control variate would be closely correlated to $\rvi{}$: the higher the correlation between $\rvi{}$ and $Z$, the greater the variance reduction. However, $\expec{Z}$ must be precisely known lest the control variate introduces a bias. Additionally, the benefit of an optimal $\rho$ comes at the cost of estimating $\var{Z}$ and $\covar{\rvi{},Z}$.

What we actually do is estimate $\expec{X}=\expec{U}+\rho\expec{Z}$, with $U:=\rvi{}-\rho Z$, as
\begin{gather}
  \label{eq:reduced-estimator}
  \theta_{n_U,n_Z}(X) = \frac{1}{n_U}\sum_{k=1}^{n_U}U(\omega_k) + \frac{\rho}{n_Z}\sum_{k=1}^{n_Z}Z(\omega'_k) \\
  \intertext{where $\omega$ and $\omega'$ are \myiid{} samples, with variance}
  \label{eq:var-sum}
  \var*{\theta_{n_U,n_Z}(X)} = \frac{\var*{U}}{n_U} +  \frac{\rho^2\var*{Z}}{n_Z}.
\end{gather}
This estimation's cost is $c_{n_U,n_Z} = n_Uc_U + n_Zc_Z$ with $c_U:=c_0+c_Z$, where $c_0:=\cost(\rvi{})$ and $c_Z:=\cost(Z)$ denote the cost of an evaluation of $\rvi{}$ and $Z$, respectively. For a given variance $\eta^2$, we minimise $c_{n_U,n_Z}$ over $(n_U,n_Z)\in\setR^2$ by solving the equivalent saddle-point problem
\begin{gather*}
  \max_{\lambda\geqslant0} \min_{n_U,n_Z\geqslant0} c_un_u+c_Zn_Z+ \lambda \mleft(\frac{\var*{U}}{n_U}+\frac{\rho^2\var*{Z}}{n_Z}-\eta^2\mright),
\end{gather*}
which yields $n_U\sqrt{c_U\rho^2\var{Z}}=n_Z\sqrt{c_Z\var{U}}$.
Therefore, equation~\cref{eq:var-sum} now reads
\begin{gather}
  \var*{\theta_{n_U,n_Z}(X)} = \frac{1}{n_Z}\mleft(\rho^2\var*{Z} +  \sqrt{\var*{U}\var*{Z}\rho^2\frac{c_U}{c_Z}}\mright) \notag\\
  \intertext{with}
  \label{eq:samples-repartition}
  n_Z = \ceil*{ \frac{1}{\eta^2}\mleft(\rho^2\var*{Z} +  \sqrt{\var*{U}\var*{Z}\rho^2\frac{c_U}{c_Z}}\mright) }
  \text{ and }
  n_U = \ceil*{ n_Z\sqrt{\frac{c_Z\var*{U}}{c_U\var*{Z}\rho^2}} }
\end{gather}
for a variance of%
\footnote{The variance is actually slightly lower, due to the ceil function $\ceil{\cdot}$.}
$\eta^2$.
Like expectations, variances and covariances are estimated empirically. For two random variables $A$ and $B$, with a sample $(\omega_i)_{i\in\Zint{1}{m}}$ of size $m\in\setNsp$, we define
\begin{gather*}
  \varsigma_m(A,B) := \frac{1}{m-1}\sum_{i=1}^m (A(\omega_i) - \theta_m(A))(B(\omega_i) - \theta_m(B)).
\end{gather*}
This is an unbiased estimator, \myie{} $\covar{A,B} = \expec{\varsigma_m(A,B)}$; obviously, $\var{A} = \expec{\varsigma_m(A,A)}$. We set a minimal sample size required to get a reliable estimation of $\var{U}$, $\var{Z}$, $c_U$ and $c_Z$ before estimating $n_Z$ and $n_U$. Likewise, the optimal control coefficient value is estimated from formula~\cref{eq:argmin-var} as  $\varsigma_{n_U}(\rvi{},Z)/\varsigma_{\max(n_U,n_Z)}(Z)$.

A cost-efficient approximation of the homogenised operator, provided by a surrogate model, would be a closely correlated random variable. In a situation where the surrogate model does not provide a satisfying approximation, it remains a suitable control variate $Z$ to a more precise, and costlier, solution $\rvi{}$. We will first present a control variate that has been introduced a few years ago in a similar context of quasi-periodicity by~\textcite{Legoll2015a}, then propose another one based on the tensor-based approximation method.

\subsection{Weakly stochastic approach}
\label{sec:weakly-stoch}

\subsubsection{Defect-type surrogate model}

The various surrogate models of weakly stochastic homogenisation have been set in the theoretical framework outlined in~\cref{sec:stoch-homog,sec:homog-map} and address randomly perturbed periodic media. They all rely on an asymptotic development of corrector functions with respect to a parameter which quantifies the weakness of the perturbation, hence the name. Different surrogate models have been proposed for different perturbation's modelling \parencites(as in)(){Anantharaman2010a}{Blanc2006}{Thomines2012}, and we are interested in one proposed for the specific context of local material defects at the microscopic scale. This \enquote{defect-type} approach, introduced by \textcite{Anantharaman2010b}, is based on the assumptions that every material cell%
\footnote{Period of the reference unperturbed material.}
is either sound or faulty (all defects being identical), that every cell has the same probability $\wikns$ of being faulty and that this probability is small.

\begin{proposition}[Defect-type perturbation]
  \label{prop:defect-type}
  Let $K\in\mathcal{M}(\alpha,\beta)$ be such that
  \begin{gather*}
    K(x,\omega) = K_{\sharp}(x) + K_{def}(x) \sum_{i\in I} \mathds{1}_{D_i}(x) B_\wikns^i(\omega), \quad\forall(x,\omega)\in\setR^d\times\Omega,
  \end{gather*}
  with $K_\sharp$ and $K_{def}$ periodic, and independent random variables $\mleft(B^i_\wikns\mright)_{i\in I}$ following the Bernoulli law of parameter $\wikns$, so that $\mathbb{P}(B^i_\wikns=1)=\wikns$. Then,
  \begin{gather*}
    K_\star =  \lim_{N\rightarrow\infty} K_{\sharp\star} + \wikns N^dK_{1\star}^N + \wikns^2 K_{2\star}^N + o(\wikns^2)	
  \end{gather*}
  where $K_{\sharp\star}$ is the periodic homogenised operator defined in \cref{prop:homog-per},
  \begin{gather*}
    \mleft(K_{1\star}^N\mright)_{ij} := \frac{1}{\abs{D}}\int_{D} (e_i+\grad w_i^{1,N}) \cdot (K_{\sharp}+\mathds{1}_{D_1}K_{def}) \cdot e_j - (K_{\sharp\star})_{ij},
  \end{gather*}
  with the $d$ correctors $w^{1,N}_i\in H^1_\sharp(D)$ solutions of
  \begin{gather*}
    -\diver\mleft((K_{\sharp}+\mathds{1}_{D_1}K_{def})\cdot(e_i+\grad w_i^{1,N})\mright) = 0,
  \end{gather*}
  and
  \begin{gather}
    \label{eq:2def-contribution}
    K_{2\star}^N = \frac{1}{2}\sum_{l\in I}\sum_{m\in I\setminus\{l\}} K_{2\star}^{1+\abs{m-l},N}
  \end{gather}
  where
  \begin{gather*}
    \mleft(K_{2\star}^{k,N}\mright)_{ij} = \frac{1}{\abs{D}}\int_{D} \mleft(e_i+\grad w_{i,k}^{2,N}\mright)\cdot (K_{\sharp}+\mathds{1}_{D_1\cup D_k}K_{def}) \cdot e_j - 2\mleft(K_{1\star}^N\mright)_{ij} + (K_{\sharp\star})_{ij},
  \end{gather*}
  with the $d\times(N^d-1)$ correctors $w_{i,k}^{2,N}\in H^1_\sharp(D)$ solutions of
  \begin{gather*}
    -\diver\mleft((K_{\sharp}+\mathds{1}_{D_1\cup D_k}K_{def})\mleft(e_i+\grad w_{i,k}^{2,N}\mright)\mright) = 0.
  \end{gather*}
\end{proposition}
\begin{proof}
  See~\textcite{Anantharaman2011b}.
\end{proof}

In \cref{prop:defect-type}, sound cells have diffusion coefficient $K_\sharp$ and faulty cells have $K_{\sharp} + K_{def}$. As an example, the conductivity expressed in~\cref{eq:K-bernoulli}, involved in most tests of \cref{sec:lrm-tests}, falls in this defect-type category. As in \cref{sec:lrm-tests}, we still consider isotropic mesoscopic grids of $N$ cells along each direction so $\#I = N^d$, and set as reference cell $Y:= D_1$, the \enquote{first} cell.

The terms $K_{1\star}^N$ and $K_{2\star}^{k,N}$ are referred to as 1-defect and 2-defects individual contributions, respectively. They are the homogenised operators associated with the periodic repetition of domain $D$ with, respectively, one and two defects. Due to periodicity, only the relative localisation of one defect with respect to the other affects the apparent homogenised operator; with only one defect, its position does not matter. Therefore, the total 1-defect contribution is $N^dK_{1\star}^N$ whereas the total 2-defects contribution is defined by equation~\cref{eq:2def-contribution}, where we chose arbitrarily to place the first defect in $D_1$. The \enquote{0-defect} contribution is, obviously, $K_{\sharp\star}$.

\subsubsection{Variance reduction based on defect-type model}
\label{sec:var-red-weakly}

Weakly stochastic homogenisation provides a surrogate model whose accuracy is related to perturbation weakness. When the perturbation is not weak enough for that model to provide the desired accuracy, its approximation remains supposedly highly correlated to the estimated homogenised operator $K_\star^N$. Therefore, it can be used as control variate in the Monte Carlo estimation $\theta_m(K_\star^N)$ (from equation~\cref{eq:mc-estimator}), as detailed in the introduction of \cref{sec:var-reduction}, where $K_\star^N$ is evaluated by FEM. This variance reduction scheme has been published by~\textcite[in][]{Legoll2015a} for the defect-type approach from \cref{prop:defect-type}. We will outline the method here and refer the reader to this publication for a detailed explanation and justification of what follows.

\Cref{prop:defect-type} offers a second-order approximation from which can be deduced either a first- or second-order controlled variable, depending on whether the second-order term is accounted for. This decision is a compromise between accuracy and cost-efficiency, which will be discussed in \cref{sec:variance-exp}. The first-order controlled variable is defined as
\begin{gather*}
  X_1 = K_\star^N - \rho_1(Z_1-\expec*{Z_1}) \\
  \intertext{with}
  Z_1 = K_{\sharp\star} + \sum_{i\in I} B_i^\wikns K_{1\star}^N, \quad 
  \text{and}\quad
  \expec*{Z_1} = K_{\sharp\star} + \wikns N^d K_{1\star}^N, \\
  \intertext{therefore, our first-order controlled variable reads}
  X_1 = K_\star^N -\rho_1\mleft( \sum_{i\in I}B_i^\wikns- \wikns N^d \mright) K_{1\star}^N \\
  \intertext{and, from~\cref{eq:argmin-var}, the optimal control coefficient value is}
  \rho_1 = \frac{\covar*{K_\star^N,Z_1}}{\var*{Z_1}}.
\end{gather*}
We used here the same notations as in \cref{prop:defect-type}.
Henceforth we will note $c_N:=\cost\mleft(K^N_\star\mright)$ and, accordingly, $c_1:=\cost(K_{\sharp*})$. The computational cost related to control variate $Z_1$ can be broken down into an a priori cost $c_{p1}$ and the cost of evaluating a new realisation of $Z_1$. The latter is deemed negligible%
\footnote{Compared to, \myeg{} $c_1$.}
while former amounts to the computation of $K_{\sharp\star}$ and $K_{1\star}^N$, hence $c_{p1}:=c_{FE}+c_1$.

The second-order controlled variable is expressed as
\begin{gather*}
  X_2 = K_\star^N - \rho_1(Z_1-\expec*{Z_1}) - \rho_2(Z_2-\expec*{Z_2}) - \check{\rho}_2\mleft(\check{Z}_2-\expec*{\check{Z}_2}\mright) \\
  \intertext{with}
  Z_2 = \frac{1}{2} \sum_{i\in I} \sum_{j\in I\setminus\{i\}}B_i^\wikns B_j^\wikns K_{2\star}^{1+\abs{j-i},N}  \text{ so }
    \expec*{Z_2} = \frac{\wikns^2}{2} \sum_{i\in I} \sum_{j\in I\setminus\{i\}} K_{2\star}^{1+\abs{j-i},N},\\
  \intertext{and}
  \check{Z}_2 = \frac{1}{2} \sum_{i\in I} \sum_{j\in I\setminus\{i\}}(1-B_i^\wikns)(1-B_j^\wikns), \check{K}_{2\star}^{1+\abs{j-i},N}  
  \text{ so }
  \expec*{\check{Z}_2} = \frac{(1-\wikns)^2}{2} \sum_{i\in I} \sum_{j\in I\setminus\{i\}} \check{K}_{2\star}^{1+\abs{j-i},N}.
\end{gather*}
We introduced above the quantity $\check{K}_{2\star}^{k,N}$ for $k\in I$. It is the counterpart of $K_{2\star}^{k,N}$ for the \enquote{complementary} material, \myie{} the material whose reference conductivity is $K_\sharp+K_{def}$ and whose faulty cells have conductivity $K_\sharp$. Consequently, the associated 2-defects individual contributions are
\begin{gather*}
  \mleft(\check{K}_{2\star}^{k,N}\mright)_{ij} := \frac{1}{\abs{D}}\int_{D} (e_i+\grad \check{w}_{i,k}^{2,N})\cdot (K_{\sharp}+K_{def}-\mathds{1}_{D_1\cup D_k}K_{def}) \cdot e_j - 2\check{K}_{1\star}^N + \check{K}_{\sharp\star} \\
  \text{with}\quad
  -\diver\mleft((K_{\sharp}+K_{def}-\mathds{1}_{D_1\cup D_k}K_{def})(e_i+\grad \check{w}_{i,k}^{2,N})\mright) = 0, \\
  \intertext{and the 1-defect individual contribution is}
  \mleft(\check{K}_{1\star}^N\mright)_{ij} := \frac{1}{\abs{D}}\int_{D} (e_i+\grad \check{w}_i^{1,N}) \cdot (K_{\sharp}+K_{def} -\mathds{1}_{D_1}K_{def}) \cdot e_j - (\check{K}_{\sharp\star})_{ij} \\
  \text{with}\quad
  -\diver\mleft((K_{\sharp}+K_{def} -\mathds{1}_{D_1}K_{def})\cdot(e_i+\grad \check{w}_i^{1,N})\mright) = 0, \\
  \intertext{and, finally, the periodically homogenised operator is}
  (\check{K}_{\sharp\star})_{ij} := \frac{1}{\abs{Y}}\int_Y (e_i+\grad \check{w}_i^\sharp)\cdot (K_\sharp+K_{def}) e_j , \\
  \text{with}\quad
  -\diver\mleft((K_\sharp+K_{def})(e_i+\grad \check{w}_i^\sharp)\mright) = 0.
\end{gather*}
This additional contribution was ignored in \cref{prop:defect-type} because we assumed $\wikns\ll 1$. There is no such assumption here and therefore no precedence of one reference over the other. This does not apply at the first order, for which considering one or the other is equivalent. Indeed, the first-order control variate accounts only for the proportion of faulty cells without geometric consideration.

As for the control coefficients $\rho_1$, $\rho_2$ and $\check{\rho}_2$, their optimal values are solutions to
\begin{gather}
  \label{eq:control-coef-2}
  \begin{pmatrix}
    \var{Z_1}             & \covar{Z_1,Z_2}       & \covar{Z_1,\check{Z}_2} \\
    \covar{Z_2,Z_1}       & \var{Z_2}             & \covar{Z_2,\check{Z}_2} \\
    \covar{\check{Z}_2,Z_1} & \covar{\check{Z}_2,Z_2} & \var{\check{Z}_2}       
  \end{pmatrix}  
  \begin{pmatrix}
    \rho_1 \\ \rho_2 \\ \check{\rho}_2
  \end{pmatrix}
  =
  \begin{pmatrix}
    \covar{K_\star^N,Z_1} \\ \covar{K_\star^N,Z_2} \\ \covar{K_\star^N,\check{Z}_2} 
  \end{pmatrix}.
\end{gather}

The a priori cost $c_{p2}$ of this second-order control variate is significantly higher than the first-order one: $c_{p2}=2(C_1+(N^d-1)C_{FE})$. The subsequent cost of every realisation is admittedly higher as well, mainly due to the resolution of~\cref{eq:control-coef-2}, yet we still deem it negligible.

Obviously, the cost of these a priori computations could be greatly reduced by MsLRM. Even more so because modes recycling would be most efficient in this case.

\subsection{Low-rank approximation as control variate}
\label{sec:eim}

Results from \cref{sec:lrm-tests} hinted at the MsLRM limits, \myie{} situations where a more general and simple direct resolution (\myeg{} FEM) would be more efficient. These situations occur mostly when the rank required to achieve the desired precision is too high, because either the tolerance is too low or the quasi-periodicity too weak. To address these situations, we wish to propose a MsLRM-based low-fidelity model as a control variate to $K_\star^N$ evaluated by a high-fidelity model---FEM here. 
First, we observed that the rank of $K$ had a significant effect on the approximation's rank. To address cases where $K$ may not be quasi-periodic, one way to provide a cost-efficient approximation would be to use a low-rank approximation of $K$. However, to guarantee the approximation's quality we must still be able to control the precision of this perturbed MsLRM.

Let $\varepsilon\in\setRsp$, we note $U^\varepsilon_i$ the function from $\mathcal{M}(\alpha,\beta)$ to $V_h(D)$ which maps $K$ to the MsLRM approximation $w^\varepsilon_{N,i}$ of the apparent corrector solution to~\cref{eq:app-corr-sto} (for a chosen $i\in\Zint{1}{d}$). We choose $\delta\in\setRsp$ and let $\qpK{}\in L^\infty(D)$ be such that $\norm{\qpK{}-K}_{L^\infty(D)}\leqslant \delta$ and $\rank{\qpK{}}$ is low. Then, assuming $U^\varepsilon_i\in C^0(L^\infty(D);V_h(D))$, we have $\norm{U^\varepsilon_i(\qpK{})-U^\varepsilon_i(K)} \leqslant C_U \delta$ where $C_U$ denotes the Lipschitz constant of $U^\varepsilon_i$. Consequently, we can control the error caused on the MsLRM approximation by the approximation on $K$, provided that we can control $\delta$.

Then, let us note $Z_{\delta,\varepsilon}$ the function from $L^2(\Omega;\mathcal{M}(\alpha,\beta))$ to 
$L^2(\Omega;L^\infty(D)^{d\times d})$ which maps $K$ to the apparent homogenised operator associated to $\qpK{}$ with MsLRM approximation to the precision $\varepsilon$ of the correctors, \myie{}
\begin{gather}
  \label{eq:app-homog_control-variate}
  \mleft(Z_{\delta,\varepsilon} (K)\mright)_{ij} := \frac{1}{\abs{D}} \int_D e_i \cdot\qpK{} (e_j+\grad U^\varepsilon_j(\qpK{})).
\end{gather}
We hope to provide a cost-efficient, highly correlated control variate $Z_{\delta,\varepsilon}(K)$ to $K_\star^N$. More precisely, we hope to achieve cost-efficiency for aperiodic conductivities $K$ with large enough $\delta$ and $\varepsilon$ while remaining highly correlated to $K_\star^N$. We note the controlled variable
\begin{gather*}
  X_{\delta,\varepsilon} := K_\star^N + \rho(Z_{\delta,\varepsilon}-\expec*{Z_{\delta,\varepsilon}}),
\end{gather*}
where control coefficient $\rho$ is computed according to expression~\cref{eq:argmin-var}.
The process that yield $X_{\delta,\varepsilon}$ from $K^N$ is summarised on \cref{fig:control-variate_diagram}.

\begin{figure}[hbt]
  \centering

\colorlet{highfid}{femFG}
\colorlet{highfidBG}{femBG}
\colorlet{lowfid}{mslrmFG}
\colorlet{lowfidBG}{mslrmBG}
\hypersetup{hidelinks}

\begin{tikzpicture}[node distance=11em,decoration={snake,pre length=.5ex,post length=.5ex,amplitude=2pt},every node/.style={rounded corners,opacity=.1,text opacity=1}]
  \node (cond) {$K^N(\omega)$};
  \node[right of=cond,fill=highfidBG] (corrRef) {$w_{N,k}(\omega)$};
  \node[right of=corrRef,fill=highfidBG] (homRef) {$K^N_{*}(\omega)$};
  \draw[->,highfid] (cond.east) -- node[above] {\scshape FEM} (corrRef.west);
  \draw[->,highfid] (corrRef.east) -- node[above] {app. hom.} node[below] {\eqref{eq:app-homog-sto}} (homRef.west);
  \node[below of=cond,node distance=4em,fill=lowfidBG] (condApprox) {$\tilde{K}^N(\omega)$};
  \node[right of=condApprox,fill=lowfidBG] (corrApprox) {$U^\varepsilon_i\left(\tilde{K}^N(\omega)\right)$};
  \node[right of=corrApprox,fill=lowfidBG] (homApprox) {$Z_{\delta,\varepsilon}(\omega)$};
  \draw[->,decorate,lowfid] (cond.south) -- node[left] {\scshape EIM} node[right] {$\delta$} (condApprox.north);
  \draw[->,lowfid,decorate] (condApprox.east) -- node[above] {\textsc{M}\textsmaller{s}\textsc{LRM}} node[below] {$\varepsilon$} (corrApprox.west);
  \draw[->,lowfid] (corrApprox.east) -- node[above] {app. hom.} node[below] {\eqref{eq:app-homog_control-variate}} (homApprox.west);
  \node[right of=homRef,opacity=1,draw=TolLightGreen] (contrVar) {$X_{\delta,\varepsilon}(\omega)$};
  \node[right of=homApprox,fill=lowfidBG] (expecCV) {$\mathds{E}\left(Z_{\delta,\varepsilon}(\omega)\right)$};
  \draw[->,highfid] (homRef) -- (contrVar) ;
  \draw[->,lowfid] (homApprox) -- (contrVar) ;
  \draw[->,lowfid] (expecCV) -- (contrVar) ;
\end{tikzpicture}


  \caption[Combination of high- and low-fidelity models into controlled variable]{Combination of \textcolor{highfid}{high-} and \textcolor{lowfid}{low-}fidelity models into controlled variable $X_{\delta,\varepsilon}$}
  \label{fig:control-variate_diagram}
\end{figure}
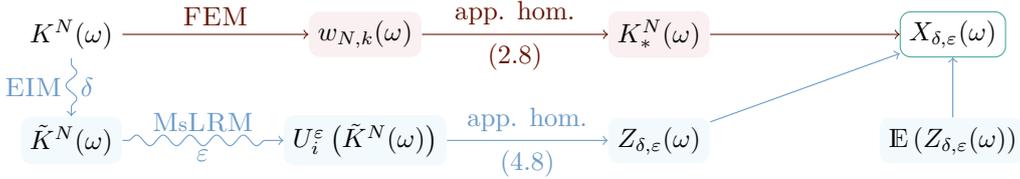

The low-rank approximation $\qpK{}$ is built as an empirical interpolation of $K$ in the form $\mathcal{I}_n[K]:=\sum_{j=1}^n \alpha_j^n\otimes q_j$ by greedy \cref{algo:eim}.
At every step $n\in\setNsp$, the interpolation points $(i_n,y_n)$ are defined to be where the current interpolation is worse.
Then $q_n$ is chosen colinear to $K(i_n,\cdot)-\mathcal{I}_{n-1}[K](i_n,\cdot)$ and such that $q_n(y_n)=1$.
Since the matrix $(q_j(y_p))_{(j,p)\in\Zint{1}{n}^2}$ is invertible%
\footnote{It is actually lower triangular, with unity diagonal~\parencite[see][th. 2.2 p. 387]{Maday2009a}.},
we can choose $\alpha_n$ such that $\forall (i,p)\in I\times\Zint{1}{n},\; \mathcal{I}_n[K](i,y_p)=K(i,y_p)$.
Once $\norm*{K-\mathcal{I}_n[K]}_{L^\infty(D)}\leqslant \delta$, we stop and keep approximation $\qpK{}:=\mathcal{I}_n[K]$.
The a priori error estimation
$\max_{i\in I} \norm*{K(i,\cdot)-\mathcal{I}_n[K](i,\cdot)}_{L^\infty(Y)} \leqslant (1+\Lambda_n) \inf_{\kappa\in\mathcal{K}_n}\norm*{K-\kappa}_{L^\infty(Y)}$
holds, with $\mathcal{K}_n:=\vspan\mleft\{\{K(i_p,\cdot), p\in\Zint{1}{n}\}\mright\}$ and the Lebesgue constant $\Lambda_n := \sup_{y\in Y}\sum_{j=1}^n \abs*{h_j^n(y)}\leqslant 2^n-1$, where $h_j^n\in\mathcal{K}_n$ is such that $\forall p\in\Zint{1}{n}, h_j^n(y_p)=\delta_{jp}$.
From the quasi-periodicity assumption on $K$, we expect $\inf_{\kappa\in\mathcal{K}_n}\norm*{K-\kappa}_{L^\infty(Y)}$ to decrease rapidly with $n$ and thus believe that \cref{algo:eim} can provide accurate and low-rank approximations in our case.
For a detailed explanation of the Empirical Interpolation Method (EIM), we refer the reader to~\cites{Barrault2004a}{Maday2009a}.

\begin{algorithm}
  \DontPrintSemicolon
  Initialise $\mathcal{I}_0[K] := 0$, $R_0:=K$ and $n:=1$\; 
  \While{$\sup_{i\in I}\norm*{R_n(i,\cdot)}_{L^\infty(Y)}>\delta$}{
    Define $R_{n-1}:= K - \mathcal{I}_{n-1}[K]$\;
    Compute points $i_n := \argmax_{i\in I}\norm*{R_{n-1}}_{L^\infty(Y)}$ and $y_n := \argmax_{y\in\overline{Y}}\abs*{R_{n-1}(i,y)}$\;
    Set $\displaystyle q_n:= \frac{R_{n-1}(i_n,\cdot)}{R_{n-1}(i_n,y_n)}$\;
    Compute $\alpha^n$ by solving $\sum_{j=1}^{n} \alpha_j^{n}(i)q_j(y_p) = K(i,y_p),\; \forall (i,p)\in I\times\Zint{1}{n}$\;
    Update $\mathcal{I}_n[K] := \sum_{j=1}^n\alpha_j^n\otimes q_j$, $R_{n}:= K - \mathcal{I}_{n}[K]$ and $n:=n+1$\;
  }
  \KwRet{$\qpK{} := \mathcal{I}_n[K]$}
  \caption{Empirical interpolation of $K$}
  \label{algo:eim}
\end{algorithm}

\begin{remark}[Multi-level control variates]
  The benefit of a control variate to an estimator is measured by the error reduction achieved for a given computational budget or, conversely, by the cost reduction achieved for a given precision.
  Therefore a control variate quality can be assessed, with formula~\cref{eq:min-var}, from its correlation with the variable of interest and its cost: the former should be high and the latter low.
  Consequently, the quality of $Z_{\delta,\varepsilon}$ hinges on the choice of the couple of tolerances $(\delta,\varepsilon)$.
  A strategy could be devised to choose good values based on measures of cost and correlation on reasonably small samples.

  Another approach to this issue would be to combine several such control variates into a variant of estimator $\theta_{n_U,n_Z}(X)$, defined by formula~\cref{eq:reduced-estimator}.
  Thus, for a number $n\in\setNsp$ of control variates, we would choose $(\delta,\varepsilon)\in\setR^n\times\setR^n$ then observe that
  \begin{align}
    X & = \rvi{} - \rho_1Z_{\delta_1,\varepsilon_1} + \rho_1\expec*{Z_{\delta_1,\varepsilon_1}} \notag\\
    \label{eq:mfcv-telescopic}
      & = \rvi{} - \rho_nZ_{\epsilon_n,\delta_n} + \rho_nZ_{\epsilon_n,\delta_n} - \rho_{n-1}Z_{\epsilon_{n-1},\delta_{n-1}} \ldots + \rho_2Z_{\epsilon_2,\delta_2} - \rho_1Z_{\epsilon_1,\delta_1} + \rho_1\expec*{Z_{\epsilon_1,\delta_1}}.
  \end{align}
  Several multi-level estimators have been based on this, such as the famous multi-level Monte Carlo methods detailed in the monograph~\cite{Heinrich2001a}.
  These have been applied successfully by~\textcites{Giles2008a}{Cliffe2011a} to stochastic partial differential equations, with demonstrated complexity reduction.
  An optimal sampling strategy could then be formulated via a similar reasoning as for the two-level case, which lead to formul\ae~\cref{eq:samples-repartition}, albeit not as straightforward.
  
  Alternatively, \textcite{Peherstorfer2016a} proposed multi-fidelity Monte Carlo estimators, which stem from the same idea.
  For a sample $\omega\in\Omega^M$ of size $M\in\setNsp$ and $n\in\setNsp$ low-fidelity models (i.e. control variates, here), such an estimator could be
  \begin{gather*}
    \frac{1}{m_0}\sum_{k=1}\rvi{}(\omega_k) +  \sum_{j=1}^{n}\mleft(\frac{1}{m_j}\sum_{k=1}^{m_j} Z_{\delta_j,\varepsilon_j}(\omega_k) - \frac{1}{m_{j-1}}\sum_{k=1}^{m_{j-1}} Z_{\delta_j,\varepsilon_j}(\omega_k)\mright)
  \end{gather*}
  where $\forall j\in\Zint{0}{M}$, $0 < m_j \leqslant M$.
  For the sake of brevity, we will note the high-fidelity model $\zeta_0:=\rvi{}$, and   $\zeta_j=Z_{\delta_j,\varepsilon_j}$ any low-fidelity model $j\in\Zint{1}{n}$; furthermore, $\forall j \in\Zint{0}{n}$, $\gamma_j := \cost(\zeta_j)$ and $r_j = \corr{\zeta_0,\zeta_j}$.
  Let us also introduce control coefficients $\rho\in\setR^n$ as we did at the beginning of \cref{sec:var-reduction}, \cpageref{eq:argmin-var}.
  We then define the multi-fidelity estimator
  \begin{gather}
    \label{eq:mfcv-def}
    \vartheta_m(X) := \frac{1}{m_0}\sum_{k=1}\zeta_0(\omega_k) +  \sum_{j=1}^{n}\rho_j\mleft(\frac{1}{m_j}\sum_{k=1}^{m_j} \zeta_j(\omega_k) - \frac{1}{m_{j-1}}\sum_{k=1}^{m_{j-1}} \zeta_j(\omega_k)\mright).
  \end{gather}
  
  Interestingly, there are results on error minimisation of such multi-fidelity estimators for a fixed budget.
  Owing to the telescopic sum%
  \footnote{And the fact that, $\forall j\in\Zint{0}{m}, m_j>0$.}
  in~\cref{eq:mfcv-telescopic} and~\cref{eq:mfcv-def}, $\vartheta_m(X)$ is an unbiased estimator of $\expec{X}$.
  Consequently, its mean squared error is
  \begin{gather*}
    \var{\vartheta_m(X)} = \frac{\var{\zeta_0}}{m_0} + \sum_{j=1}^n \mleft(\frac{1}{m_{j-1}}-\frac{1}{m_j}\mright) (\rho_j^2\var{\zeta_j}-2\rho_j\covar*{\zeta_0,\zeta_j}),
  \end{gather*}
  which we note $E(m,\rho)$, and $\cost(\vartheta_m(X)) = \sum_{j=0}^nm_j\gamma_j$.

  Given $(\delta,\varepsilon)$, estimator $\vartheta_m(X)$ is uniquely defined by the choice of sample sizes $m\in\setNsp^{n+1}$ and control coefficients $\rho$.
  For a given computational budget $B\in\setRsp$, an optimal choice with respect to the estimation error would be solution to
  \begin{gather}
    \label{eq:mfcv-optimal}
    \min_{m\in\setRsp^{n+1},\; \rho\in\setR^n} \mleft\{ E(m,\rho) : \cost(\vartheta_m(X))= B \mright\}.
  \end{gather}
  Let us make the following assumptions%
  \footnote{\label{fn:corr-conv}With the convention $r_{n+1}:=0$.},
  $\forall j \in \Zint{1}{n-1}$:
  \begin{align}
    m_j \geqslant m_{j+1} \; ; \notag\\
    \label{eq:hyp-mfcv-optimal_corr}
    \abs{r_j} > \abs{r_{j+1}} \; ;\\
    \label{eq:hyp-mfcv-optimal_cost-corr}
    \frac{\gamma_{i}}{\gamma_{i+1}} > \frac{r_i^2 - r_{i+1}^2}{r_{i+1}^2 - r_{i+2}^2}.
  \end{align}
  Under these conditions, \textcite[see][th. 3.4]{Peherstorfer2016a} proved that problem~\cref{eq:mfcv-optimal} has a unique solution $(\breve{m},\breve{\rho})$ expressed as
  \begin{gather*}
    \breve{m}_0 = B \mleft(\sum_{j=0}^n \sqrt{\gamma_j\gamma_0\frac{r_j^2 - r_{j+1}^2}{1-r_1^2}}  \mright)^{-1}, \\
    \forall j\in\Zint{1}{n},\quad \breve{m}_j = \breve{m}_0 \sqrt{\frac{\gamma_0}{\gamma_j}\frac{r_j^2 - r_{j+1}^2}{1-r_1^2}} \quad
    \text{and}\quad \breve{\rho}_j = \frac{\covar*{\rvi{},Z_j}}{\var{Z_j}}.
  \end{gather*}
  Then the mean squared error (i.e. variance, here) reduction achieved is\cref{fn:corr-conv}
  \begin{gather}
    \label{eq:mfmc_variance-reduction}
    \frac{E(\breve{m},\breve{\rho})p}{\var{\rvi{}}} = \mleft(\sum_{j=0}^{n} \sqrt{\frac{\gamma_j}{\gamma_0} (r_j^2 - r_{j+1}^2)} \mright)^2,    
  \end{gather}
  where $p:=B/\cost(\rvi{})$. If $p\in\setNsp$ and $\breve{m}\in\setNsp^{n+1}$ (see below), the left-hand side of~\cref{eq:mfmc_variance-reduction} is $\var{\theta_{\breve{m}}(X)}/\var{\theta_p(\rvi{})}$ and $\cost(\theta_{p}(\rvi{})) = \cost(\theta_{\breve{m}}(X)) = B$.
  
  In conditions~\cref{eq:hyp-mfcv-optimal_corr,eq:hyp-mfcv-optimal_cost-corr}, we see that only cost and correlation are involved; particularly, no assumption is made as to the pointwise accuracy of the approximation yielding the control variates.
  Moreover, the error reduction expressed in equation~\cref{eq:mfmc_variance-reduction} shows that the contribution of one control variate depends on its correlation to the others.
  A strategy could be developed to choose $(\delta,\varepsilon)\in\setRsp^{n}\times\setRsp^{n}$ so as to satisfy conditions~\cref{eq:hyp-mfcv-optimal_corr,eq:hyp-mfcv-optimal_cost-corr}, in which case we would know the optimal choice $(\breve{m},\breve{\rho})$.
  
  It should be noted that minimisation problem~\cref{eq:mfcv-optimal} is posed over $\setRsp^{n+1}\times\setR^n$; consequently, $\breve{m}$ will have to be projected onto $\setNsp^{n+1}$.
  For a given computational budget not to be exceeded, we would use the lower integers $\floor*{\breve{m}}$, whereas we would take $\ceil*{\breve{m}}$ in order to satisfy a given tolerance $t$ on the mean squared error.
  In the latter case, \textcite[see][th. 1]{Peherstorfer2018b} has proven, under some assumptions, that $\cost(\vartheta_m(X)) \in \mathcal{O}(t^{-1})$; a convergence rate which, according to~\textcite{Cliffe2011a}, is also achieved by multi-level Monte Carlo methods.

  Finally, \textcite{Peherstorfer2017a} proposes an adaptive approach to account for the non-negligible creation cost of a control variate, whereas the aforementioned methods only consider the sampling cost.
  Its purpose is to balance the overall budget between construction and evaluation of a single surrogate model, so that the model design is cost-oriented as well.
  This too could be adapted to the method proposed here, although there is not so much a \enquote{construction-evaluation}%
  \footnote{Sometimes referred to as \enquote{offline-online}.}
  cost separation as an \enquote{initialisation-asymptotic} one; the former refers to the library construction cost and the latter to the evaluation cost once the library has reached practical completion.
\end{remark}

\subsection{Numerical results}
\label{sec:variance-exp}

We consider bidimensional problems of type~\cref{eq:diff-strong} and choose three different random stationary conductivities $K$. The first one is an ideal quasi-periodic case previously introduced in \cref{sec:lrm-tests}, with a conductivity of rank at most $2$; it falls within the scope of the control variate technique detailed in \cref{sec:var-red-weakly} and allows us to compare it to the one from \cref{sec:eim}. Both other conductivities where chosen to assess the limits of this latter technique and have almost surely full rank%
\footnote{\myIe{} $r_K=N^2$.}.
The first one still follows a regular mesoscopic grid whereas the last one exhibits no such structure.

For each test case we wish to estimate the total cost reduction from computing $\theta^\eta(K^N_\star)$ to computing $\theta^\eta(X)$ for different controlled versions $X$ of $K^N_\star$ and always for the same standard deviation $\eta:=\num{0.1}$. Those total costs are respectively $C_{FE}=\hat{c}_n\var{K^N_\star}/\eta^2$, and $C_X=(\hat{c}_n+c_Z)n_U + c_Zn_Z$ where $\hat{c}_n$ and $c_Z$ are the costs of a single estimation of $K^N_\star$ and of the control variate%
\footnote{Either $Z_1$, $Z_2$ or some $Z_{\delta,\varepsilon}$, depending on the chosen surrogate model.}
$Z$ involved in $X$, respectively; $n_U$ and $n_Z$ are calculated according to formul\ae~\cref{eq:samples-repartition}. Single costs, variances and covariances are estimated from \num{100} samples of $K^N_\star$, $X$ and $Z$. $X_{\delta,\varepsilon}$ is tested with five values of $(\delta,\varepsilon)$ on each test case, and $X_1$ and $X_2$ are tested on the first test case.

All tests are set on a domain $D:=[0,20]^2\equiv\Zint{1}{400}\times[0,1]^2$. Control variates are evaluated by MsLRM with $\mathrm{dim}(V_h(Y)):=441$ and EIM cost is accounted for, albeit lower than \SI{0.1}{\second} per interpolation. $K^N_\star$ is computed by FEM with the same mesh as MsLRM, thus \num{160801} degrees of freedom%
\footnote{Slightly less than $400\times441=\num{176400}$, due to continuity.};
its associated cost $\hat{c}_n\approx\SI{246}{\second}$ does not vary significantly from one test case to another.

These computations were run on the same hardware as described at the beginning of \cref{sec:hardware}, \cpageref{sec:hardware}.

\subsubsection{Quasi-periodic inclusions}
\label{sec:qp-inclusions}

The first case is defined in equation~\cref{eq:K-bernoulli} and illustrated on \cref{fig:bi-periodic-K}; we set defect probability $p:=0.5$ for greatest periodicity loss. \Cref{tab:var-red-1} displays the cost reductions achieved with various control variates, along with values on which this cost reduction depends. This table has been designed so that greater values mean better performance, except for the average ranks $r_w$ and $r_{\qpK{}}$, respectively for correctors and empirically interpolated conductivities. The single cost reduction $\hat{c}_n/c_Z$ was irrelevant for $Z_1$ and $Z_2$ whose evaluation cost is almost zero, so we display instead $\hat{c}_n/c_{p}$ between parentheses, where $c_{p}$ is the a priori cost discussed in \cref{sec:var-red-weakly} (either $c_{p1}$ or $c_{p2}$, respectively); those computations were performed by MsLRM with tolerance $\varepsilon:=\num{0.01}$. The other values between parentheses are the optimal variance reductions, estimated from formula~\cref{eq:min-var}; the difference between those and the observed reductions comes from the uncertainties in empirical estimations of variances and covariances.

\begin{table}
  \centering
  \caption{Cost and variance reductions on quasi-periodic inclusions --- $\var{K^N_\star}=\num{1.8}$}
  \begin{tabular}{S[table-format=1.2]
    S[table-format=2.2]
    S[table-format=2.1]
    S[table-format=1]
    S[table-format=3.2]
    S[table-format=5]
    rS[table-format=4]
    S[table-format=3.1]}\toprule
    {$\delta$} & {$\varepsilon$} & {$r_{\qpK{}}$} & {$r_w$} & {$\hat{c}_n/c_Z$} & \multicolumn{2}{c}{$\var{K^N_\star}/\var{X}$} & {$n_Z/n_U$} & {$C_{FE}/C_X$} \\\midrule
    {--$Z_1$--} & 0.01 & {--} & 1 & \multicolumn{1}{c}{(\tablenum[table-format=3.2]{59})} & 10747 & (\num{10747}) & 3337 & 180 \\
    {--$Z_2$--} & 0.01 & {--} & 11 & \multicolumn{1}{c}{(\tablenum[table-format=3.2]{0.11})} & 15673 &  & 461 & 178 \\
    1 & 1 & 1 & 1 & 165 & 10751 & (\num{11335}) & 209 & 79.3\\
    1 & 0.1 & 1 & 1 & 164 & 10751 & (\num{11335}) & 209 & 79.1\\
    1 & 0.01 & 1 & 7.2 & 134 & 2054 & (\num{2392}) & 208 & 70.5\\
    0.1 & 1 & 2 & 1 & 140 & 15327 & (\num{15503}) & 208 & 72.3\\
    0.01 & 1 & 2 & 1 & 140 & 15327 & (\num{15503}) & 208 & 72.3\\
    \bottomrule
  \end{tabular}
  \label{tab:var-red-1}
\end{table}

The first observation is that all control variates achieve comparable variance reduction, and therefore the major difference is in their cost. Hence $Z_1$ and $Z_2$ are the most efficient by far and the cost reduction displayed is only limited by the fact that we always do at least one computation of $K^N_\star$ (\myie{} $n_U\geqslant1$). Here the variance reduction is such that every control variate achieved $n_U=1$, and $Z_1$ and $Z_2$ reduced the total cost to that single computation: $C_X\approx \hat{c}_n$. This ignores the a priori cost which, although low for $Z_1$, is much greater for $Z_2$ and makes it questionable whether the additional variance reduction is worth the extra cost%
\footnote{Ignoring the influence of the complementary reference periodic material would half this extra cost but reduce correlation.}.
These a priori costs were significantly reduced by using MsLRM instead of FEM: the values of $\hat{c}_n/59$ and $9\hat{c}_n$ would have been, respectfully, $2\hat{c}_n$ and $2(N^2+1)\hat{c}_n=802\hat{c}_n$ with FEM.

As for $Z_{\delta,\varepsilon}$, $\varepsilon$ has little effect (except a cost increase) but $\delta$ makes the difference between a rank-1 approximation $\qpK{}$, which yields the same variance reduction as the first-order control variate $Z_1$, and the exact rank-2 conductivity from which we get the same variance reduction as with the second-order control variate $Z_2$. There is hardly any difference in cost between the cases where $r_{\qpK{}}=1$ and those where $r_{\qpK{}}=2$, $c_Z$ varies only from \SIrange{1.5}{2}{\second}. Because every $Z_{\delta,\varepsilon}$ reaches the optimal situation $n_U=1$ here, the differences in total cost reduction $C_{FE}/C_X$ is not significant and one should rather gauge performance from the variance reduction $\var{K^N_\star}/\var{X}$.

\subsubsection{Regular peaks}
\label{sec:regular-peaks}

The second conductivity tested exhibits a mesoscopic structure following a regular grid, yet it is aperiodic in the sense that its rank almost surely equals the number of cells. Its expression is
\begin{gather*}
  K(i,y,\omega) = 1 + \alpha_i(\omega) e^{-\norm*{\beta_i(\omega)(y-\gamma)}}  
\end{gather*}
where $\alpha_i\in \mathcal{U}([0,199])$ and $\beta_i=\mathrm{diag}((\beta_{i,1},\beta_{i,2}))$ with $(\beta_{i,1},\beta_{i,2})\in\mathcal{U}([2,10])^2$. The exponential peak is centred within each cell with $\gamma:=(0.5, 0.5)$. An example of this function is displayed on \cref{fig:variance-example-2}. Its contrast (between highest and lowest values) and variance are on a par with the previous conductivity, yet its periodicity is much lower.

\begin{table}
  \centering
  \caption{Cost and variance reductions on regular peaks --- $\var{K^N_\star}=\num{1.2}$}
  \begin{tabular}{S[table-format=1.2]
    S[table-format=2.2]
    S[table-format=2.1]
    S[table-format=2.1]
    S[table-format=2.2]
    S[table-format=4]
    rS[table-format=3]
    S[table-format=2.1]}\toprule
    {$\delta$} & {$\varepsilon$} & {$r_{\qpK{}}$} & {$r_w$} & {$\hat{c}_n/c_Z$} & \multicolumn{2}{c}{$\var{K^N_\star}/\var{X}$} & {$n_Z/n_U$} & {$C_{FE}/C_X$} \\\midrule
    1 & 1 & 22.7 & 1 & 29 & 15 & (\num{17}) & 16 & 5.9\\
    1 & 0.1 & 22.7 & 23.9 & 8.29 & 1307 & (\num{1367}) & 66 & 6.6\\
    0.1 & 1 & 35.3 & 1 & 19 & 15 & (\num{18}) & 13 & 5.0\\
    0.01 & 1 & 49.5 & 1 & 14 & 15 & (\num{18}) & 11 & 4.3\\
    \bottomrule
  \end{tabular}
  \label{tab:var-red-2}
\end{table}

Results are presented in \cref{tab:var-red-2} in the same format as \cref{tab:var-red-1} previously explained, yet without $Z_1$ and $Z_2$ which cannot be used here. We first observe that a cost reduction is achieved for every value of $(\delta,\varepsilon)$ tested but $(1,0.01)$. The results for this latter case are not available since the approximation reached ranks so high ($r_w\geqslant100$) that the test machine ran out of memory%
\footnote{Memory usage went over \SI{8}{\gibi\byte}.};
these computations were costlier than the FEM ones and serve to point out MsLRM limits. Then we see that a low tolerance $\delta$ yields better results since a high conductivity is quite detrimental to MsLRM performance and does produce a better correlated control variate. Indeed, \num{23} is already a great rank increase from the previous test case, although much better than the raw $r_K=\num{400}$. Finally, performance improves when $\varepsilon$ is lowered just enough%
\footnote{Here, a better compromise could be found with $\varepsilon\in]0.1,1[$.}
to go beyond the crude rank-1 approximation of the correctors.

\begin{figure}
  \centering
  \begin{subfigure}{.48\linewidth}
    \centering
    \includegraphics[height=\linewidth,width=\linewidth]{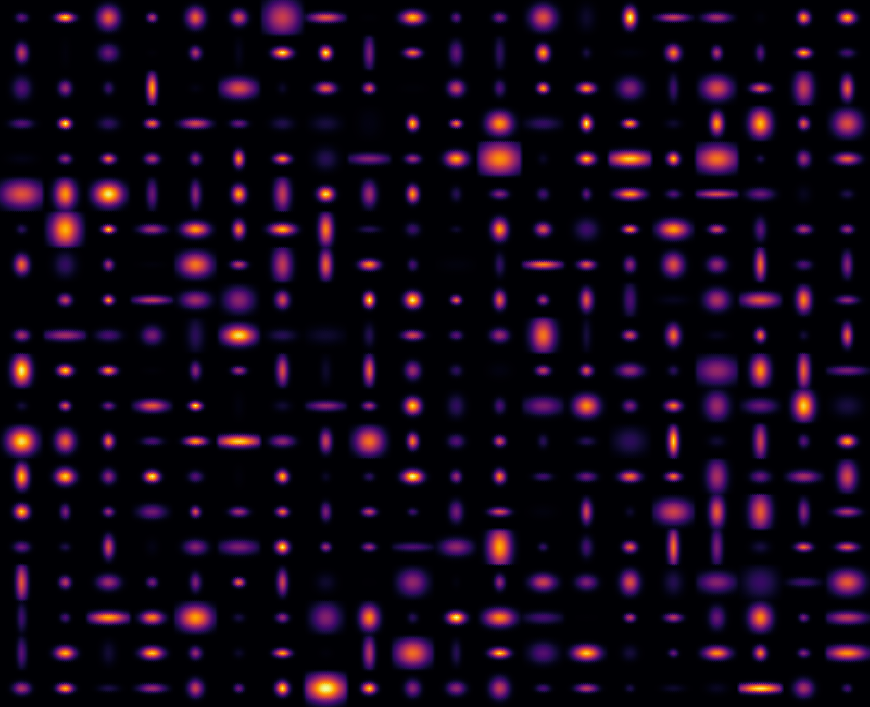}
    \caption{Regular peaks}
    \label{fig:variance-example-2}
  \end{subfigure}
  ~  
  \begin{subfigure}{.48\linewidth}
    \centering
    \includegraphics[height=\linewidth,width=\linewidth]{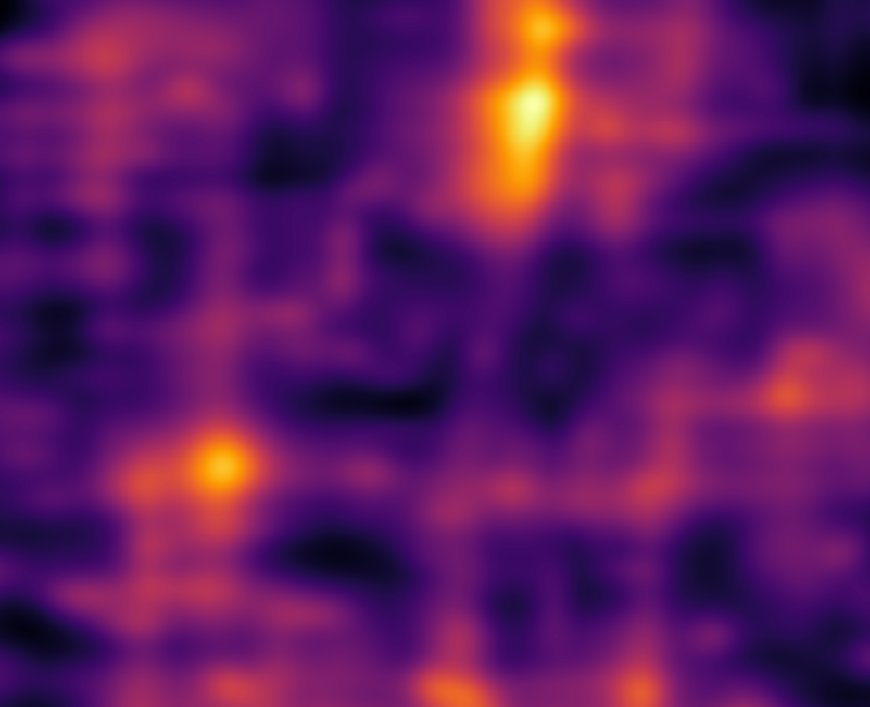}
    \caption{Irregular peaks}
    \label{fig:variance-example-3}
  \end{subfigure}
  \caption{Examples of second and third test conductivity functions on a $20\times 20$ grid}
  \label{fig:variance-examples}
\end{figure}

\subsubsection{Irregular peaks}
\label{sec:irregular-peaks}

This last conductivity function has been chosen to be an irregular variant of the second one, with similar exponential peaks distributed according to a stationary Poisson point process. It reads
\begin{gather*}
  K(x,\omega) = 1 + \sum_{n=1}^{\nu(\omega)} \alpha_n(\omega) e^{-\norm*{\beta_n(\omega)(x-\gamma_n(\omega))}}
\end{gather*}
where  the families $(\alpha_n)$ and $(\beta_n)$ are composed of independent and identically distributed random variables: $\alpha_n\in\mathcal{U}([0, 179])$ and $\beta_n=\mathrm{diag}((\beta_{n,1},\beta_{n,2}))$ with $(\beta_{n,1},\beta_{n,2})\in\mathcal{U}([0.1,2])^2$. The number of peaks $\nu$ follows a Poisson distribution of expectation $(N+R)^2$. $R$ is defined as the radius beyond which a peak's value is below tolerance $\delta$, hence
\begin{gather*}
  \norm*{\alpha_n e^{-\norm*{\beta_n\times\tensor[^t]{(R, R)}{}}}}_{L^\infty(\Omega)}\leqslant\delta , \\
  \text{so } R:=\lambda_{min}(\beta_n)^{-1}\ln\mleft(\frac{\norm*{\alpha_n}_{L^\infty(\Omega)}}{\delta}\mright)
  =10\ln\mleft(\frac{179}{\delta}\mright) \\
  \text{with } \lambda_{min} = \min_{\omega\in\Omega} \min \mleft\{\beta_{n,1}(\omega),\beta_{n,2}(\omega)\mright\}.
\end{gather*}
The centres $\gamma_n$ are independent random variables following the same uniform law $\mathcal{U}([-R,N+R]^2)$. This represents the truncation to $D:=[0,N]^2$ of a spatial Poisson point process over $\setR^2$, accounting for a radius $R$ around each point.
An example of such $K$ is illustrated on \cref{fig:variance-example-3}.

\begin{table}
  \centering
  \caption{Cost and variance reductions on irregular peaks --- $\var{K^N_\star}=\num{3642}$}
  \begin{tabular}{S[table-format=1.2]
    S[table-format=2.2]
    S[table-format=3.1]
    S[table-format=2.1]
    S[table-format=3.2]
    S[table-format=5]
    rS[table-format=4]
    S[table-format=3.1]}\toprule
    {$\delta$} & {$\varepsilon$} & {$r_{\qpK{}}$} & {$r_w$} & {$\hat{c}_n/c_Z$} & \multicolumn{2}{c}{$\var{K^N_\star}/\var{X}$} & {$n_Z/n_U$} & {$C_{FE}/C_X$} \\\midrule
    1 & 1 & 28.7 & 19.8 & 11.1 & 4628 & (\num{8226}) & 237 & 10.0\\
    1 & 0.1 & 28.7 & 45.4 & 3.15 & 4624 & (\num{8276}) & 139 & 3.0\\
    1 & 0.01 & 28.7 & 72.9 & 1.13 & 4623 & (\num{8275}) & 100 & 1.1\\
    0.1 & 1 & 41.8 & 19.6 & 9.42 & 4822 & (\num{8796}) & 225 & 8.6\\
    \bottomrule
  \end{tabular}
  \label{tab:var-red-3}
\end{table}

The results in \cref{tab:var-red-3} follow the same layout as the previous one. Surprisingly, the results are somewhat better than those reported in \cref{sec:regular-peaks}, even though the average of $r_{\qpK{}}$ is slightly higher. As before, every control variate tested reduces the overall cost but $Z_{1,0.01}$, which does not offer further variance reduction than $Z_{1,0.1}$. Unlike the previous case, $Z_{1,1}$ goes beyond the rank-1 approximations of the correctors and is the best control variate here since it achieves almost as good a variance reduction as $Z_{1,0.1}$. We conclude that a crude empirical interpolation of the conductivity along with low-rank (but above the trivial rank 1) approximation of the correctors yields significant cost-reduction (factor \num{10} for $Z_{1,1}$).

Another important parameter to discuss here would be the mesoscopic discretisation. Although the choice was obvious for every other test case presented, it was arbitrary for this one. A different cell size based, for example, on the radius $R$ might be more relevant.

\section{Conclusion}
\label{sec:conclusion}

We proposed a way to reduce the cost of stochastic homogenisation for quasi-periodic conductivities by a low-rank approximation of the corrector functions identified with bivariate functions. The subsequent error on the homogenised quantities has been found to be reasonable on our examples, whereas the computational cost is greatly reduced in most cases, compared to FEM. The complexity was reduced first regarding domain size and then to samples number. The major influence on the cost of this multiscale low-rank method~\parencite[MsLRM, from][]{Ayoul-Guilmard2018a} comes from the rank of the conductivity identified as a tensor; the closer to a periodic function, the lower the rank.

Interestingly, the number of occurrences of a given defect does not change the conductivity rank and therefore has no effect on the complexity of the associated corrector problem. More generally, materials that would not be intuitively considered quasi-periodic have been found to be of low complexity for this MsLRM. A category of such materials is quasi-periodic ones transformed through a random mapping with low rank gradient. Another one is materials following a regular grid and whose grid cells patterns are linear combinations of the same small number of patterns.

This reduced complexity for aperiodic media with mesoscopic structure was exploited by building a low-rank empirical interpolation of the conductivity, computing MsLRM approximations of the associated corrector functions and using the resulting homogenised operator as a control variate to the homogenised operator computed by direct FEM. This was found to provide cost-efficient and well correlated control variates in our numerical experiment, not only for quasi-periodic conductivities but also for materials with not enough periodicity for the MsLRM to be efficient. We illustrated and discussed the influences of tolerances on both empirical interpolation on the conductivity and low-rank approximation of the correctors. Additionally, the MsLRM reduced greatly the a priori cost of the weakly stochastic control variate from~\textcite{Legoll2015a}, which is more specific but less expensive.

Finally, this variance reduction technique, applied here to quantities computed by FEM, could be used in an experimental campaign on measured quantities, with sufficient knowledge of the material to deduce a low-rank approximation of the conductivity as described in \cref{sec:eim}.

\section*{Acknowledgement}

The authors gratefully acknowledge the financial support from the Fondation CETIM.

\printbibliography{}

\end{document}